\documentclass{amsart}

\usepackage{latexsym,amssymb,amsthm,amsmath,amscd}

\theoremstyle{plain}
\newtheorem{theorem}{Theorem}[section]
\newtheorem*{theorem A}{Theorem A}
\newtheorem*{theorem B}{Theorem B}

\newtheorem{lemma}{Lemma}[section]
\newtheorem{proposition}{Proposition}[section]

\newtheorem{corollary}{Corollary}[section]

\newtheorem{definition}{Definition}[section]
\newtheorem{example}{Example}[section]
\numberwithin{equation}{section}

\theoremstyle{remark}
\newtheorem{remark}{Remark}[section]

 \numberwithin{equation}{section}

\newtheorem*{Theorem A}{{\bf Theorem A}}
\newtheorem*{Theorem B}{{\bf Theorem B}}
\newtheorem*{Theorem C}{Theorem C}

 \numberwithin{equation}{section}
\def\<{\left < }
\def\>{\right >}
\def\({\left ( }
\def\){\right )}

\begin{document}

\title[Vanishing Theorem]{On vanishing theorems for vector bundle valued $p$-forms and their applications}

\author[Yuxin Dong]{Yuxin Dong$^*$}
\address{Institute of Mathematics\\
\newline {Fudan University, Shanghai 200433}\\
\newline {P.R. China}\\
\newline {And}\\
\newline {Key Laboratory of Mathematics}\\
\newline {for Nonlinear Sciences}\\
\newline {Ministry of Education}}
\email{yxdong@fudan.edu.cn}
\author[Shihshu Walter Wei]{Shihshu Walter Wei$^{**}$}
\address{Department of Mathematics\\
\newline {University of Oklahoma}\\
\newline {Norman, Oklahoma 73019-0315}\\
\newline {U.S.A.}}
\email{wwei@ou.edu}

\thanks {$^*$Supported by NSFC grant No 10971029, and NSFC-NSF grant
No 1081112053 \\
$^{**}$ Research was partially supported by NSF Award No DMS-0508661, OU Presidential International Travel Fellowship,
and OU Faculty Enrichment Grant.}

\maketitle





\begin{abstract}{ Let $F: [0, \infty) \to [0, \infty)$ be a strictly increasing $C^2$ function with $F(0)=0$. We unify the concepts of $F$-harmonic maps, minimal hypersurfaces, maximal spacelike hypersurfaces, and Yang-Mills Fields, and introduce $F$-Yang-Mills fields, $F$-degree, $F$-lower degree, and generalized Yang-Mills-Born-Infeld fields (with the plus sign or with the minus sign) on manifolds. When $F(t)=t\, , \frac 1p(2t)^{\frac p2}\, , \sqrt{1+2t} -1\, ,$ and
$1-\sqrt{1-2t}\, ,$  the $F$-Yang-Mills field becomes an ordinary Yang-Mills field, $p$-Yang-Mills field, a generalized Yang-Mills-Born-Infeld field with the plus sign, and a generalized Yang-Mills-Born-Infeld field with the minus sign on a manifold respectively. We also introduce the
 $E_{F,g}-$energy functional (resp. $F$-Yang-Mills functional) and derive the first variational formula of the $E_{F,g}-$energy functional (resp. $F$-Yang-Mills functional)   with applications. In a more general frame, we use a unified method to study the stress-energy tensors that arise from calculating the rate of change of various functionals when the metric of the domain or base manifold is changed. These stress-energy tensors are naturally linked to $F$-conservation laws and yield monotonicity
formulae, via the coarea formula and comparison theorems in Riemannian geometry.
Whereas a ``microscopic" approach to some of these monotonicity formulae leads to celebrated blow-up techniques and regularity theory in geometric measure theory, a ``macroscopic" version of these monotonicity inequalities enables us to derive
some Liouville type results and vanishing theorems for $p-$forms with values in vector
bundles, and to investigate constant Dirichlet boundary value problems for $1$-forms. In particular, we obtain Liouville theorems for
$F-$harmonic maps (which include harmonic maps, $p$-harmonic maps, exponentially harmonic maps, minimal graphs and maximal space-like hypersurfaces, etc), $F-$Yang-Mills fields, extended Born-Infeld fields, and generalized Yang-Mills-Born-Infeld fields (with the plus sign and with the minus sign) on manifolds etc.  As another consequence, we obtain the unique constant solution of the constant Dirichlet boundary value problems on starlike domains for vector bundle-valued $1$-forms satisfying an $F$-conservation law, generalizing and refining the work of Karcher and Wood on
harmonic maps. We also obtain generalized Chern type results for constant
mean curvature type equations for $p-$forms on $\Bbb{R}^m$ and on manifolds $M$ with the global doubling property by a different approach. The case $p=0$ and $M=\mathbb{R}^m$ is due to Chern.}
\end{abstract}

\section{Introduction}

A theorem due to Garber, Ruijsenaars, Seiler and Burns [GRSB]
states that every harmonic map $u:$ $\Bbb{R}^m\rightarrow S^m$
with finite energy must be constant($m>2$). This result has been
generalized by Hildebrandt [Hi] and Sealey [Se1] to
harmonic maps into arbitrary Riemannian manifolds from more general domains, for example from an hyperbolic $m$-space form, or from $\Bbb{R}^m$ with certain globally conformal
flat metrics, where $m > 2\, .$
In the context of harmonic
maps, the stress-energy tensor was introduced and studied in
detail by Baird and Eells [BE]. Following Baird-Eells [BE],
Sealey [Se2] introduced the stress-energy tensor for vector
bundle valued $p-$forms and established some vanishing theorems
for $L^2$ harmonic $p-$forms. Liouville
type theorems for vector bundle valued
harmonic forms or forms satisfying certain conservation laws have been treated by [KW] and [Xi1]. These follow immediately from monotonicity formulae.
A similar technique was also used by [EF1] and [EF2] to show
nonexistence of $L^2-$eigenforms of the Laplacian (on functions
and differential forms) on certain complete noncompact manifolds
of nonnegative sectional curvature.

On the other hand, in [Ar], M. Ara introduced the $F-$harmonic map and its associated
stress-energy tensor. Let $F:[0,\infty )\rightarrow [0,\infty )$
be a $C^2$
function such that $F^{\prime }>0$ on $[0,\infty )\, ,$ and $F(0)=0$. A smooth map $%
u:M\rightarrow N$ between two Riemannian manifolds is said to be an
\emph{$F-$harmonic map} if it is a critical point of the following
$F-$energy functional $E_F$ given by
\begin{equation}
E_F(u)=\int_MF(\frac{|du|^2}2)dv  \tag{1.1}
\end{equation}
with respect to any compactly supported variation, where $|du|$ is
the Hilbert-Schmidt norm of the differential $du$ of $u$, and $dv$ is the volume element of $M$. When
$F(t)=t$, $\frac 1p(2t)^{\frac p2}$, $(1+2t)^{\alpha}\, (\alpha >1, \dim M=2)\, ,$ and $e^t$, the $F-$harmonic
map becomes a harmonic map, a $p-$harmonic map, an $\alpha$-harmonic map, and an
exponentially harmonic map respectively. One of these striking features is that we can use, for example $p$-harmonic maps to study topics or problems that do not seem to be approachable by ordinary harmonic maps (in which $p=2$)(see e.g. [We2,3,LWe]).

In addition to the above examples, $F-$energy functionals and their critical points arise widely in
geometry and physics. Recall that a minimal hypersurface in
$\Bbb{R}^{m+1}$, given as the graph of the function $u$ on a
Euclidean domain satisfies the following differential equation and is a solution of  Plateau's problem (for any closed $m-1$-dimensional submanifold in the minimal graph as a given boundary):
\begin{equation}
\operatorname{div}\left( \frac{\nabla u}{\sqrt{1+|\nabla u|^2}}\right) =0  \tag{1.2}
\end{equation}
If a maximal spacelike hypersurface in Minkowski space $\mathbb{R}^{n,1}$ $\big ($ with the coordinate $(t, x^1,\cdots,x^n)$ and the metric $ds^2 = dt^2 - \sum_{i=1}^n (dx^i)^2\, \big )$ is given as the graph of the
function $v$ on a Euclidean domain, then the function $v$
satisfies
\begin{equation}
\operatorname{div}\left( \frac{\nabla v}{\sqrt{1-|\nabla v|^2}}\right) =0\, .  \tag{1.3}
\end{equation}
Obviously the solutions $u$ and $v$ are $F-$harmonic maps from a domain in
$\Bbb{R}^m$ to $\Bbb{R}$ with $F=\sqrt{1+2t}$ $-1$ and
$F=1-\sqrt{1-2t}$ respectively, with respect to any compactly supported variation. In [Ca], Calabi showed that
equations (1.2) and (1.3) are equivalent over any simply connected
domain in $\Bbb{R}^2$. Along the lines of Calabi, Yang [Ya] showed
that, for $m=3$, equations (1.2) and (1.3) over a simply connected
domain are, respectively, equivalent instead to the vector
equations
\begin{equation}
\nabla \times \left( \frac{\nabla \times A}{\sqrt{1\mp |\nabla \times A|^2}}
\right) =0  \tag{1.4}
\end{equation}
(where $A$ is a vector field in $\Bbb{R}^3$ and $\nabla \times (\, \cdot\, )\, $ is the curl of $(\, \cdot\, )$ ) which arise in the
nonlinear electromagnetic theory of Born and Infeld [BI]. This
observation leads Yang [Ya] to give a generalized treatment of
equations of (1.2) and (1.3) expressed in terms of differential
forms as follows:
\begin{equation}
\delta \left( \frac{d\omega }{\sqrt{1+|d\omega |^2}}\right) =0,\qquad \omega
\in A^p(\mathbb{R}^m)  \tag{1.5}
\end{equation}
and
\begin{equation}
\delta \left( \frac{d\sigma }{\sqrt{1-|d\sigma |^2}}\right)
=0,\qquad \sigma \in A^q(\mathbb{R}^m)  \tag{1.6}
\end{equation}
(where $d$ is the exterior differential operator and $\delta$ is the codifferential operator), and a reformulation of Calabi's equivalence theorem in arbitrary
$n$ dimensions. Born-Infeld theory is of contemporary interest due
to its relevance in string theory ([BN], [DG], [Ke], [LY], [Ya],
[SiSiYa]). It is easy to verify that the solutions of (1.5) and
(1.6) are critical points of the following Born-Infeld type
energy functionals
\begin{equation}
E_{BI}^{+}(\omega )=\int_{\mathbb{R}^m} \sqrt{1+|d\omega |^2}-1\quad dv
\tag{1.7}
\end{equation}
and
\begin{equation}
E_{BI}^{-}(\sigma )=\int_{\mathbb{R}^m} 1-\sqrt{1-|d\sigma |^2}\quad dv
\tag{1.8}
\end{equation}
respectively. By choosing a sequence of cutoff functions and
integrating by parts, Sibner-Sibner-Yang [SiSiYa] established a
Liouville theorem for the $L^2$ exterior derivative $d\omega$ of a solution $\omega$ of (1.5). They also introduced Yang-Mills-Born-Infeld fields and
obtained a Liouville type result for finite-energy solutions of a
generalized self-dual equation reduced from the
Yang-Mills-Born-Infeld equation on $\Bbb{R}^4$.

In this paper, we unify the concepts of $F$-harmonic maps, minimal hypersurfaces in Euclidean space, maximal spacelike hypersurfaces in Minkowski space, and Yang-Mills Fields, and introduce $F$-Yang-Mills fields, $F$-degree, $F$-lower degree, and generalized Yang-Mills-Born-Infeld fields (with the plus sign or with the minus sign) on manifolds(cf. Definitions 3.2, 4.1, 6.1 and 8.1). When $F(t)=t\, ,$ $\frac 1p(2t)^{\frac p2}\, ,$ $\sqrt{1+2t} -1\, ,$ and
$1-\sqrt{1-2t}\, ,$  the $F$-Yang-Mills field becomes an ordinary Yang-Mills field, a $p$-Yang-Mills field, a generalized Yang-Mills-Born-Infeld field with the plus sign, and a generalized Yang-Mills-Born-Infeld field with the minus sign on a manifold respectively. We also introduce the
 $E_{F,g}-$energy functional (resp. $F$-Yang-Mills functional) and derive the first variational formula of the $E_{F,g}-$energy functional (resp. $F$-Yang-Mills functional) (Lemmas 2.5 and 3.1) with applications. In a more general frame, we use a unified method to study the stress-energy tensors that arise from calculating the rate of change of various functionals when the metric of the domain or base manifold is changed. These stress-energy tensors lead to a fundamental integral
formula (2.10), are naturally linked to $F$-conservation laws. For example, we prove that every $F-$Yang-Mills
field satisfies an $F$-conservation law. In particular, every $p-$Yang-Mills
field satisfies a $p$-conservation law (cf. Theorem 3.1 and Corollary 3.1). As an immediate consequence, the simplified integral formula (2.11), from (2.10) holds
for vector bundle valued forms satisfying an $F-$conservation law in general, and holds for  $F-$Yang-Mills
field in particular. This yields monotonicity
inequalities, via the coarea formula and comparison theorems in Riemannian geometry (cf. Theorem 4.1 and Proposition 4.1). Whereas a ``microscopic" approach to monotonicity formulae leads to celebrated blow-up techniques due to E. de-Giorgi [Gi] and W.L. Fleming [Fl], and regularity theory in geometric measure theory(cf. [FF,A,SU,PS,HL,Lu]). For example, the regularity results of Allard [A] depend on the monotonicity formulae for varifolds. The regularity results of Schoen and
Uhlenbeck [SU] depend on the monotonicity formulae for harmonic
maps which they derived for energy minimizing maps;
monotonicity properties are also dealt with by Price and Simon [PS] for
Yang-Mills fields, and by Hardt-Lin [HL] and Luckhaus [Lu] for $p$-harmonic maps.
A ``macroscopic" version of these monotonicity formulae enable us to
derive some Liouville type results and vanishing theorems under suitable growth conditions on
Cartan-Hadamard manifolds or manifolds which possess a pole with
appropriate curvature assumptions (e.g. Theorems 5.1 and 5.2). In particular, our results are
applicable to $F-$harmonic maps, $F-$Yang-Mills fields, extended Born-Infeld fields, and generalized Yang-Mills-Born-Infeld fields (with the plus sign or with the minus sign) on manifolds, and obtain the first vanishing theorem for $p$-Yang-Mills fields (cf. Theorems 5.3-5.8).
In fact, we introduce the following
$E_{F,g}-$energy functional
\begin{equation}
E_{F,g}(\sigma )=\int_MF(\frac{|d^\nabla \sigma |^2}2) dv_g\tag{2.12}
\end{equation}
for forms $\sigma \in A^{p-1}(\xi )$ with values in a Riemannian vector bundle $\xi$,
or study an even more general functional $\mathcal{E}_{F,g}(\omega )$
for forms $\omega \in A^p(\xi )$ (see (2.5)), introduced by Lu-Shen-Cai [LSC].
Naturally,
the stress-energy tensor associated with
$E_{F,g}(\sigma )$ or $\mathcal{E}_{F,g}(\omega )$ plays an
important role in establishing Liouville type results for extremals of $E_{F,g}$ or forms satisfying an $F-$conservation
law.

Our growth assumptions in Liouville type theorems in the general settings (cf. (5.1), (5.4), Theorems 5.1 and 5.2) are weaker than the assumption of \emph {finite energy} for harmonic maps due to Garber, Ruijsenaars, Seiler and Burns [GRSB], Sealey [Se1], and others, or \emph {finite $F$-energy} for $F$-harmonic maps due to M. Kassi [Ka], or \emph {$L^p$ growth} for vector bundle valued forms due to J.C. Liu [Li1], or \emph {the slowly divergent
$F-$energy condition}(e.g. (5.3)) for harmonic maps and Yang-Mills fields that was first introduced by H.S. Hu in [Hu1,2], for $F$-harmonic maps due to Liao and Liu [LL2],  and for an extremal of $\mathcal{E}_{F,g}$-energy functional treated by M. Lu, W.W. Shen and K.R. Cai [LSC](see Theorem 10.1, Examples 10.1 and 10.2 in Appendix).

Furthermore, our estimates in the monotonicity formulae are sharp in the sense that in special cases, they recapture the monotonicity formulae of harmonic maps [SU] and
Yang-Mills field [PS] (cf. Corollary 4.1. and Remark 4.2).

In addition to establishing vanishing theorems and Liouville
type results, the monotonicity formulae may be used to investigate
the constant Dirichlet boundary-value problem as well. We obtain the unique constant solution of the constant Dirichlet boundary value problem on starlike domains for vector bundle-valued $1$-forms satisfying an $F$-conservation law (cf. Theorem 6.1), generalizing and refining the work of Karcher and Wood on
harmonic maps [KW]. Notice that our constant boundary-value result holds for any
starlike domain, while the original result in [KW] was stated for
a disc domain. For an extended
Born-Infeld field $\omega \in A^p(\Bbb{R}^m)\, $ with the plus sign,  we give an upper bound of the Born-Infeld type energy $E_{BI}^{+}(\omega;G(\rho))$ of the $p$-form $\omega$ over its ``graph" $G(\rho)$ in $\mathbb{R}^{m+k}$ (cf. Proposition 7.1). This recaptures the volume estimate for the
minimal graph of $f$ due to P. Li and J.P. Wang, when $\omega =f\in A^0(\Bbb{R}^m)=C^\infty
(\Bbb{R}^m)$ (cf. [LW]).

As further applications,
we obtain vanishing theorems for extended Born-Infeld fields (with the plus sign or with the minus sign) on manifolds under an appropriate growth condition on $E_{BI}^{\pm}$-energy, and for generalized Yang-Mills-Born-Infeld fields (with the plus sign or with the minus sign)  on manifolds under an appropriate growth condition on $\mathcal{YM}_{BI}^{\pm}$-energy.
(cf. Theorems 7.1, 8.1, and 8.2, Propositions 7.2 and 8.1). The case $M=\mathbb{R}^m$ and $d \omega \in L^2\, ,$ where $\omega$ is a Born-Infeld field (hence $\omega$ has finite $E_{BI}$-energy, by the inequality $\sqrt {1+t^2} -1 \le \frac {t^2}2$ for any $t \in \mathbb{R}$)
is due to L. Sibner, R. Sibner and Y.S. Yang(cf. [SiSiYa]).

Being motivated by the work in [We1,2] and [LWW], we consider constant mean curvature type
equations for $p-$forms on $\Bbb{R}^m$ and thereby obtain generalized Chern type results for constant
mean curvature type equations for $p-$forms on $\Bbb{R}^m$ and on manifolds with the global doubling property by a different approach(cf. Theorems 9.1-9.4). The case $p=0$ and $M=\mathbb{R}^m$ is due to Chern (cf. Corollary 9.1).

This paper is organized as follows.  Generalized $F$-energy functionals and $F$-conservation laws are given in section 2. In section 3, we introduce $F$-Yang-Mills fields. In section 4, we derive monotonicity formulae.  Liouville type results and vanishing theorems are established in three subsections 5.1-5.3 of section 5. In section 6, we treat constant Dirichlet Boundary-Value Problems for vector valued $1$-forms. Extended Born-Infeld fields and exact forms are presented in section 7. In section 8, we introduce generalized Yang-Mills-Born-Infeld
fields (with the plus sign and with the minus sign) on manifolds. Generalized Chern type results on manifolds are investigated in sections 9. In the last section, we provide an appendix of a
theorem on $\mathcal{E}_{F,g}$-energy growth.

Throughout this paper
let $F: [0, \infty) \to [0, \infty)$ be a strictly increasing $C^2$ function with $F(0)=0$, and
let $M$ denote a smooth $m-$dimensional
Riemannian manifold (mostly $m>2$); all data will be assumed
smooth for simplicity unless otherwise indicated.

\section{ Generalized $F$-energy Functionals and $F$-Conservation Laws}

Let $(M,g)$ be a smooth Riemannian manifold. Let $\xi :E\rightarrow M$ be a smooth Riemannian vector bundle over $(M,g)\, ,$ i.e. a vector bundle such that at each fiber is equipped with a positive inner product $\langle \quad , \quad \rangle_E\, .$
 Set $A^p(\xi )=\Gamma (\Lambda
^pT^{*}M\otimes E)$ the space of smooth $p-$forms on $M$ with
values in the vector bundle $\xi :E\rightarrow M$. The {\it exterior
differential operator} $d^\nabla :A^p(\xi )\rightarrow
A^{p+1}(\xi )$ relative to the connection $\nabla ^E$ is given by
\begin{equation}
\aligned
 d^\nabla \sigma \,
(X_1,...,X_{p+1})=&\sum_{i=1}^{p+1}(-1)^{i+1}\nabla
_{X_i}^E(\sigma (X_1,...,\widehat{X}_i,...,X_{p+1})) \\
&+\sum_{i<j}(-1)^{i+j}\sigma ([X_i,X_j],X_1,...,\widehat{X}_i,...,\widehat{X}%
_j,...,X_{p+1})\endaligned \tag{2.1}
\end{equation}
where the symbols covered by $\, \, \widehat{}\, \, $ are omitted. Since the
Levi-Civita connection on $TM$ is torsion-free, we also have
\begin{equation}
(d^\nabla \sigma
)(X_1,...,X_{p+1})=\sum_{i=1}^{p+1}(-1)^{i+1}(\nabla _{X_i}\sigma
)(X_1,...,\widehat{X}_i,...,X_{p+1})  \tag{2.2}
\end{equation}

For two forms $\omega ,\omega ^{\prime }\in A^p(\xi )$, the
induced inner product is defined as follows:
\begin{equation}
\aligned
 \langle\omega ,\omega ^{\prime }\rangle&=\sum_{i_1<\cdots <i_p}\langle\omega
(e_{i_1},...,e_{i_p}),\omega ^{\prime }(e_{i_1},...,e_{i_p})\rangle_E \\
&=\frac 1{p!}\sum_{i_1,...,i_p}\langle\omega
(e_{i_1},...,e_{i_p}),\omega ^{\prime
}(e_{i_1},...,e_{i_p})\rangle_E
\endaligned\tag{2.3}
\end{equation}
where $\{e_1, \cdots e_m\}$ is a local orthonormal frame field on $(M,g)\, .$ Relative to the Riemannian structures of $E$ and $TM$, the
{\it codifferential operator} $\delta ^\nabla :A^p(\xi )\rightarrow
A^{p-1}(\xi )$ is characterized as the adjoint of $d$ via the
formula
$$
\int_M\langle d^\nabla \sigma ,\rho \rangle dv_g=\int_M\langle
\sigma ,\delta ^\nabla \rho\rangle dv_g
$$
where $\sigma \in A^{p-1}(\xi ),\rho \in A^p(\xi )$ , one of which
has compact support, and $dv_g$ is the volume element associated with the metric $g$ on $TM\, .$ Then
\begin{equation}
(\delta ^\nabla \omega )(X_1,...,X_{p-1})=-\sum_i(\nabla
_{e_i}\omega )(e_i,X_1,...,X_{p-1})  \tag{2.4}
\end{equation}

For $\omega \in A^p(\xi )$, set $|\omega|^2 = \langle \omega, \omega \rangle$ defined as in $(2.3)\, .$ The  authors of [LSC] defined the
following $\mathcal{E}_{F,g}$-energy functional given by
\begin{equation}
\mathcal{E}_{F,g}(\omega )=\int_MF(\frac{|\omega |^2}2)dv_g  \tag{2.5}
\end{equation}
where $F:[0,+\infty )\rightarrow
[0,+\infty )$ is as before. For our purpose, we also
allow the domain of $F$ to be $[0,c)\, ,$ where $c$ is a
positive number. In fact, we will study the case
$F:[0,\frac 12)\rightarrow [0,1)$ in Section 7.

The {\it stress-energy associated with the $\mathcal{E}_{F,g}$-energy functional} is
defined as follows (cf. [BE], [Ba], [Ar], [LSC]):
\begin{equation}
S_{F,\omega }(X,Y)=F(\frac{|\omega |^2}2)g(X,Y)-F^{\prime }(\frac{|\omega |^2%
}2)(\omega \odot \omega )(X,Y)  \tag{2.6}
\end{equation}
where $\omega \odot \omega $ denotes a $2-$tensor defined by:
\begin{equation}
(\omega \odot \omega )(X,Y)=\langle i_X\omega ,i_Y\omega \rangle
\tag{2.7}
\end{equation}
Here $\langle\quad , \quad \rangle$ is the induced inner product on $A^{p-1}(\xi)\, ,$ and $i_X\omega$ is the interior multiplication by the vector
field $X$ given by
$$
(i_X\omega)(Y_1,\ldots,Y_{p-1})=\omega(X,Y_1,\ldots,Y_{p-1})
$$
for $\omega \in A^{p}(\xi)$ and any vector fields $Y_{l}$ on $M$,
$1\leq l\leq p-1$.
When $F(t)=t$ and $\omega =du$ for a map
$u:M\rightarrow N$, $S_{F,\omega }$ is
just the {\it stress-energy tensor} introduced in [BE].

For two $2-$tensors $T_1,T_2\in \Gamma (\otimes ^2T^{*}M)$, their
inner product is defined as follows:
\begin{equation}
\langle T_1,T_2\rangle=\sum_{i,j}T_1(e_i,e_j)T_2(e_i,e_j)
\tag{2.8}
\end{equation}
where $\{e_i\}$ is an orthonormal basis with respect to $g$.

Suppose $M$ is a complete Riemannian manifold. We calculate the rate of change of the
$\mathcal{E}_{F,g}$-energy integral $\mathcal{E}_{F,g}(\omega )$ when the metric $g$ on the domain or base manifold is changed. To this end, we consider a compactly supported smooth one-parameter variation of the metric $g\, ,$ i.e. a smooth family of metrics $g_s$ such that $g_0=g\, .$  Set $\delta g =\partial g_s /\partial
s |_{s =0}\, .$ Then $\delta g$ is a smooth $2$-covariant symmetric tensor field on
$M$ with compact support.
\begin{lemma}
For $\omega \in A^p(\xi )$ $($$p\geq 1$$)$, then
$$
\frac{d\mathcal{E}_{F,g_s }(\omega )}{ds
}|_{s =0}=\frac 12\int_M\langle S_{F,\omega },\delta g
\rangle dv_g
$$
where $S_{F,\omega }$ is as in (2.6).
\end{lemma}
\begin{proof} From [Ba], we know that
$$
\frac{d|\omega |_{g_s }^2}{ds
}|_{s =0}=-\langle\omega \odot \omega ,\delta g\rangle
$$
and
$$
\frac d{ds } \, dv_{g_s}\,
_{|_{s =0}}=\frac 12\langle g,\delta g \rangle dv_g
$$
Then
$$
\aligned \frac{d\mathcal{E}_{F,g_s }(\omega )}{ds }|_{s
=0} &=\int_MF^{\prime }(\frac{|\omega |^2}2)\frac d{ds
} \big(\frac{|\omega |_{g_s }^2}2\big) |_{s =0} dv_g
+\int_MF(\frac{|\omega |^2}2)\frac d{ds } \, dv_{g_s}\,
_{|_{s =0}} \\
&=\frac 12\int_M\langle F(\frac{|\omega |^2}2)g-F^{\prime
}(\frac{|\omega |^2}2)\omega \odot \omega ,\delta g \rangle dv_g \\
&=\frac 12\int_M\langle S_{F,\omega },\delta g \rangle dv_g
\endaligned$$
\end{proof}
\begin{remark} When $F(t)=t$, the above result was derived by
Sanini in [San] and by Baird in [Ba].
\end{remark}
For a vector field $X$, we denote by $\theta _X$ its dual one
form, i.e., $ \theta _X(\cdot)=g(X,\cdot )$. By definition, the
$2$-tensor $\nabla \theta _X$ is given by
\begin{equation}
\aligned
(\nabla \theta _X)(Y,Z)&=(\nabla _Y\theta _X)(Z) \\
&=Y(\theta _X(Z))-\theta _X(\nabla _YZ) \\
&=g(\nabla _YX,Z)
\endaligned\tag{2.9}
\end{equation}
\begin{lemma} $($cf. $[Xi1]$$)$
$$
\aligned \nabla _X(\frac{|\omega |^2}2)&=\langle i_Xd^\nabla
\omega +d^\nabla i_X\omega ,\omega \rangle-\langle\omega \odot
\omega ,\nabla\theta_X\rangle\\
\langle d^\nabla i_X\omega ,\omega \rangle&=\sum_{j_1<\cdots
<j_{p-1}{; i}}\langle\omega (e_i,e_{j_1},...,e_{j_{p-1}}),(\nabla
_{e_i}\omega)(X,e_{j_1},...,e_{j_{p-1}})\rangle\\
&\qquad + \langle\omega \odot \omega,\nabla \theta _X\rangle\endaligned
$$
\end{lemma}
Next, we have the following result in which $F(t)=t$ is known in [Se2] and [Xi1]:
\begin{lemma}
 Let $\omega \in
A^p(\xi )$ $($$p\geq 1$$)$ and let $S_{F,\omega }$ be the
stress-energy tensor defined by $($2.6$)$, then for any vector field
$X$ on $M$, we have
$$
\aligned (\operatorname{div} S_{F,\omega })(X)&=F^{\prime }(\frac{|\omega
|^2}2)\langle\delta ^\nabla \omega ,i_X\omega \rangle+F^{\prime
}(\frac{|\omega |^2}2)\langle i_Xd^\nabla \omega ,\omega \rangle
\\
&\qquad - \langle i_{\text{grad}(F^{\prime }(\frac{|\omega |^2}2))}\omega
,i_X\omega \rangle
\endaligned
$$
where $\text{grad} \, (\, \bullet\, )$ is the gradient vector field of $\, \bullet\, .$
\end{lemma}

\begin{proof}By using Lemma 2.2 and (2.9), we derive the following
$$
\aligned (\operatorname{div} S_{F,\omega })(X)& = \sum_{i=1}^m\quad \nabla _{e_i}S_{F,\omega
}(e_i,X)-S_{F,\omega
}(e_i,\nabla _{e_i}X) \\
& = \sum_{i=1}^m\quad \nabla _{e_i}\big(F(\frac{|\omega |^2}2)\langle
e_i,X\rangle-F^{\prime }(\frac{|\omega |^2}
2)\langle i_{e_i}\omega ,i_X\omega \rangle\big) \\
&\qquad \qquad \quad -F(\frac{|\omega |^2}2)\langle e_i,\nabla _{e_i}X\rangle
+F^{\prime }(\frac{|\omega |^2}
2)\langle i_{e_i}\omega ,i_{\nabla _{e_i}X}\omega \rangle \\
& = \sum_{i=1}^m\quad e_i F(\frac{|\omega |^2}2)\langle e_i,X\rangle - e_i(F^{\prime
}(\frac{|\omega |^2} 2))\langle i_{e_i}\omega ,i_X\omega
\rangle\\&\qquad \qquad \quad -F^{\prime }(\frac{|\omega |^2} 2)e_i \langle
i_{e_i}\omega ,i_X\omega \rangle +F^{\prime }(\frac{|\omega
|^2}2)\langle i_{e_i}\omega ,i_{\nabla
_{e_i}X}\omega \rangle\\
&=\nabla _XF(\frac{|\omega |^2}2) - \sum_{i=1}^m\quad e_i(F^{\prime }(\frac{|\omega
|^2} 2))\langle i_{e_i}\omega ,i_X\omega \rangle\\& \qquad \qquad \qquad \qquad -F^{\prime
}(\frac{|\omega |^2} 2)e_i \langle i_{e_i}\omega ,i_X\omega
\rangle +F^{\prime }(\frac{|\omega |^2}2)\langle i_{e_i}\omega
,i_{\nabla
_{e_i}X}\omega \rangle\\
&=F^{\prime }(\frac{|\omega |^2}2)\langle i_Xd^\nabla \omega
+d^\nabla i_X\omega ,\omega \rangle-F^{\prime }(\frac{|\omega
|^2}2)\langle\omega \odot
\omega ,\nabla\theta _X\rangle \\
&\qquad -\langle i_{\text{grad}(F^{\prime }(\frac{|\omega |^2}2))}\omega
,i_X\omega
\rangle+F^{\prime }(\frac{|\omega |^2}2)\langle\delta ^\nabla \omega ,i_X\omega \rangle \\
&\qquad -F^{\prime }(\frac{|\omega |^2}2)\sum_{j_1<\cdots <j_{p-1};i}
\langle\omega (e_i,e_{j_1},...,e_{j_{p-1}}),(\nabla
_{e_i}\omega)(X,e_{j_1},...,e_{j_{p-1}})\rangle \\
&=F^{\prime }(\frac{|\omega |^2}2)\langle i_Xd^\nabla \omega
+d^\nabla
i_X\omega,\omega \rangle-\langle i_{\text{grad}(F^{\prime }(\frac{|\omega |^2}2))}\omega ,i_X\omega \rangle \\
&\qquad +F^{\prime }(\frac{|\omega |^2}2)\langle\delta ^\nabla \omega
,i_X\omega \rangle-F^{\prime }(\frac{|\omega |^2}2)\langle
d^\nabla i_X\omega ,\omega \rangle\endaligned
$$
\end{proof}

\begin{definition}
$\omega \in
A^p(\xi )$ ($p\geq 1$) is said to satisfy an
\emph {$F-$conservation law} if $S_{F,\omega }$ is divergence free, i.e. the $(0,1)-$type tensor field $\operatorname{div} S_{F,\omega }$ vanishes identically $($$\operatorname{div} S_{F,\omega }\equiv 0$$)$.
\end{definition}

\begin{lemma} $([$Ba$])$ Let $T$ be a symmetric $(0,2)-$type
tensor field. Let $X$ be a vector field, and $\theta _X$ be its dual $1$-form, then
$$
\operatorname{div} (i_XT)=(\operatorname{div} T)(X)+\langle T,\nabla \theta _X\rangle
$$
\end{lemma}

\begin{proof} Let $\{e_i\}$ be a local orthonormal frame field.
Then
$$
\aligned
\operatorname{div} (i_XT) &=\sum_{i=1}^m\big(\nabla _{e_i}(i_XT)\big)(e_i) \\
&=\sum_{i=1}^m\big(\nabla _{e_i}(T(X,e_i))-T(X,\nabla _{e_i}e_i)\big) \\
&=\sum_{i=1}^m(\nabla _{e_i}T)(X,e_i)+\sum_{i=1}^mT(\nabla _{e_i}X,e_i) \\
&=(\operatorname{div}T)(X)+\sum_{i,j=1}^mT(e_i,e_j)\, g(\nabla
_{e_i}X,e_j)\endaligned
$$
This via (2.9) proves the Lemma. \end{proof}

Let $D$ be any bounded domain of $M$ with $C^1-$boundary. By
applying $T = S_{F,\omega }$ to Lemma 2.4 and using Stokes'
Theorem, we immediately have the following
\begin{equation}
\int_{\partial D}S_{F,\omega }(X,\nu )ds_g = \int_D \langle S_{F,\omega },\nabla \theta
_X\rangle+(\operatorname{div} S_{F,\omega })(X) \quad dv_g  \tag{2.10}
\end{equation}
where $\nu $ is unit outward normal vector field along $\partial
D$ with $(m-1)$-dimensional volume element $ds_g$. In particular, if $\omega $ satisfies an $F-$conservation law, we
have
\begin{equation}
\int_{\partial D}S_{F,\omega }(X,\nu ) ds_g = \int_D\langle S_{F,\omega
},\nabla \theta _X\rangle dv_g \tag{2.11}
\end{equation}
It should be pointed out that the formulae (2.10) and (2.11) were also derived in
[LSC]. We will give some important
applications of (2.11) later.

Now we introduce a new $E_{F,g}$-energy functional
as follows: For $\sigma \in A^{p-1}(\xi )$

\begin{equation}
E_{F,g}(\sigma )=\int_MF(\frac{|d^\nabla \sigma |^2}2) dv_g\tag{2.12}
\end{equation}
This functional includes the functionals for $F-$harmonic maps (in
which $\sigma $ is a map between two Riemannian manifolds), and
Born-Infeld fields (in which $\sigma$ is the potential of an
electric field or a magnetic field and $M=\Bbb{R}^3$; cf. [Ya]) as
its special cases, etc.

\begin{lemma}[The First Variation Formula for
$E_{F,g}$-energy functional]
$$
\frac{dE_{F,g}(\sigma _t)}{dt}|_{t=0}=-\int_M\langle\tau _F(\sigma
),\eta \rangle dv_g
$$
for any $\eta \in A^{p-1}(\xi )$ with compact support, where $\sigma _t = \sigma + t\eta\, $ and $\tau _F(\sigma )=-\delta ^\nabla (F^{\prime
}(\frac{|d^\nabla \sigma |^2}2)d^\nabla \sigma )$. Furthermore, the Euler-Lagrange equation of $E_{F,g}$ is
\begin{equation}
F^{\prime }(\frac{|d^\nabla \sigma |^2}2)\tau (\sigma
)+i_{\text{grad} (F^{\prime}(\frac{|d\sigma |^2}2) )}d^\nabla \sigma =0
\tag{2.13}
\end{equation}
where $\tau (\sigma )=-\delta ^\nabla d^\nabla \sigma $.
\end{lemma}

\begin{proof} We compute
$$
\aligned \frac{dE_{F,g}(\sigma +t\eta )}{dt}|_{t=0}
&=\int_MF^{\prime }(\frac{|d^\nabla \sigma |^2}2)\langle d^\nabla \sigma ,d^\nabla \eta \rangle dv_g\\
&=\int_M\langle\delta ^\nabla (F^{\prime }(\frac{|d^\nabla \sigma
|^2}
2)d^\nabla \sigma ),\eta \rangle dv_g\\
&=-\int_M\langle\tau _F(\sigma ),\eta \rangle dv_g\endaligned
$$
where
$$
\aligned \tau _F(\sigma ) &=-\delta ^\nabla (F^{\prime
}(\frac{|d^\nabla \sigma |^2}
2)d^\nabla \sigma ) \\
&=\sum_{i=1}^m \nabla _{e_i}\big( F^{\prime }(\frac{|d^\nabla \sigma |^2}2)d^\nabla
\sigma\big)(e_i,\cdots,\cdot ) \\
&=\sum_{i=1}^m e_i(F^{\prime }(\frac{|d^\nabla \sigma |^2}2))d^\nabla \sigma
(e_i,\cdots,\cdot)+F^{\prime }(\frac{|d^\nabla \sigma
|^2}2)(\nabla_{e_i}d^\nabla \sigma )(e_i,\cdots,\cdot) \\
&=F^{\prime }(\frac{|d^\nabla \sigma |^2}2)\tau (\sigma
)+i_{\text{grad} (F^{\prime }(\frac{|d\sigma |^2}2))}d^\nabla
\sigma\endaligned
$$
\end{proof}

From Lemma 2.3 and the above expression (2.13) for $\tau _F(\sigma
)$, we immediately have the following

\begin{corollary} For $\sigma \in A^{p-1}(\xi )$, we have
$$
(\operatorname{div} S_{F,d^\nabla \sigma })(X)=-\langle\tau _F(\sigma
),i_Xd^\nabla \sigma \rangle+F^{\prime }(\frac{|d^\nabla \sigma
|^2}2)\langle i_X(d^\nabla )^2\sigma ,d^\nabla \sigma \rangle
$$
In particular, if $\tau _F(\sigma )=0$ and $(d^\nabla )^2\sigma
=0$, then $\operatorname{div} S_{F,d^\nabla \sigma }=0$.
\end{corollary}

\begin{remark} In some cases, the condition $(d^\nabla
)^2\sigma =0$ is satisfied automatically. For example, if $\sigma \in A^{p-1}(M) := \Gamma (\Lambda ^{p-1} T^{\ast}M)\, ,$ or $\sigma =d\varphi \in
A^1(\varphi ^{-1}TN)$, where $\varphi :M\rightarrow N$ is a smooth
map, then we have $(d^\nabla )^2\sigma =0$.
\end{remark}

\begin{corollary}$([$BE$]$, $[$Ka$])$ Let $\varphi :M\rightarrow N$ be an $F-$harmonic map. Then $%
\operatorname{div} S_{F,d\varphi }=0$. In particular, if $F(t)=t$ and $\varphi
:M\rightarrow N$ is a harmonic map, we have $\operatorname{div} S_{\operatorname{Id},d\varphi
}=0$.
\end{corollary}

\section{$F$-Yang-Mills Fields}
In this section we introduce $F$-Yang-Mills functionals and $F$-Yang-Mills fields.
Just as $F-$harmonic maps play a role in the space of maps between Riemannian manifolds, so do $F-$%
Yang-Mills fields in the space of curvature tensors (associated with connections on the adjoint bundles of principal $G$-bundles) over Riemannian manifolds. Let $P$ be a principal bundle
with compact structure group $G$ over a Riemannian manifold $M$.
Let $Ad(P)$ be the adjoint bundle
\begin{equation}
Ad(P)=P\times _{Ad}\mathcal{G} \tag{3.1}
\end{equation}
where $\mathcal{G} $ \ is the Lie algebra of $G$. Every connection $\rho$
on $P$ induces a connection $\nabla$ on $Ad(P)$. We also have the
Riemannian connection $\nabla ^M$ on the tangent bundle $TM$, and
the induced connection on $\Lambda ^pT^{*}M\otimes Ad(P)$. An
$Ad_G$ invariant inner product on $\mathcal{G}$ induces a fiber metric on
$Ad(P)$ and making $Ad(P)$ and $\Lambda ^pT^{*}M\otimes Ad(P)$
into Riemannian vector bundles. Although $\rho$ is not a section
of $\Lambda ^1T^{*}M\otimes Ad(P)$ , via its induced connection $\nabla $, the associated curvature $ R^\nabla$, given by $R^{\nabla}_{X,Y} = [\nabla_X, \nabla_Y]- \nabla_{[X,Y]}\, ,$ is
in $A^2(Ad(P))$. Let $\mathcal{C}$ be the space of connections $\nabla \, $ on $Ad(P)\, .$ We now introduce
\begin{definition}
The \emph{$F-$Yang-Mills
functional} is the mapping $\mathcal{YM}_F : \mathcal{C}\to \mathbb{R}^+\, $ given by
\begin{equation}
\mathcal{YM}_F(\nabla )=\int_MF(\frac 12||R ^\nabla ||^2) dv \tag{3.2}
\end{equation}
where the norm is defined in terms of the Riemannian metric on $M$ and a fixed $Ad_G$-invariant inner product on the Lie algebra $\mathcal{G}$ of $G\, .$ That is, at each point $x\in M\, ,$ its norm
$$||R^{\nabla}||^2_x = \sum_{i<j}||R^{\nabla}_{e_i,e_j}||^2_x $$
where $\{e_1, \cdots, e_n\}$ is an orthonormal basis of $T_x(M)$ and the norm of $R^{\nabla}_{e_i,e_j}$ is the standard one on Hom$(Ad(P), Ad(P))$-namely, $\langle S, U\rangle \equiv \, \text {trace}\, (S^T \circ U)\, .$
\end{definition}
\begin{definition}
A connection $\nabla$ on the adjoint bundle $Ad(P)$ is said to be an \emph{$F$-Yang-Mills connection} and its associated curvature tensor $R^{\nabla}$ is said to be an \emph{$F$-Yang-Mills field}, if $\nabla$ is a critical point of $\mathcal{YM}_F$ with respect to any compactly supported variation in the space of connections on $Ad(P)$ . A connection $\nabla$ is said to be a \emph{$p$-Yang-Mills connection} and its associated curvature tensor $R^{\nabla}$ is said to be a \emph{$p$-Yang-Mills field}, if $\nabla$ is a critical point of the $p$-Yang-Mills functional $\mathcal{YM}_p$ with respect to any compactly supported variation, where $\mathcal{YM}_p (\nabla) = \frac 1p \int _M |R^{\nabla}|^p \, dv\, ,$ and $p \ge 2$.
\end{definition}
\begin{lemma}[The First Variation Formula for
$F$-Yang-Mills functional $\mathcal{YM}_F$]\quad
Let $A \in A^1(Ad(P))$ and $\nabla ^t
=\nabla +t A $ be a family of connections on $Ad(P)$. Then
$$
\aligned \frac d{dt}\mathcal{YM}_F(\nabla ^t)|_{t=0} &=\int_M\langle\delta ^\nabla
\big(F^{\prime }(\frac 12||R
^\nabla ||^2)R ^\nabla \big), A \rangle \, dv
\endaligned$$
Furthermore, The Euler-Lagrangian equation for $\mathcal{YM}_F$ is
\begin{equation}
F^{\prime }(\frac 12||R ^\nabla ||^2)\delta ^\nabla R
^\nabla -i_{\text{grad} \big(F^{\prime }(\frac 12||R^\nabla ||^2)\big)}R
^\nabla =0  \tag{3.3}
\end{equation}
or \[
\delta ^\nabla  \big(F^{\prime }(\frac 12||R ^\nabla
||^2)R^\nabla  \big) = 0\]
\end{lemma}

\begin{proof} By assumption, the curvature of $\nabla^t$ is given by
$$
 R ^{\nabla^t} =  R ^\nabla + t (d^\nabla A) + t^2 [A, A]
$$
where $[A, A]\in A^2(Ad(P))$ is given by $[A, A]_{X,Y}=[A_X, A_Y]\, .$ Indeed, for any local vector fields $X,Y$ on $M$. with $[X,Y]=0\, ,$ we have
$$
\begin{aligned}
 R ^{\nabla^t}_{X,Y} &=(\nabla _X +t A_X)(\nabla_Y + t A_Y)-(\nabla _Y + t A_Y)(\nabla _X + t A_X)\\
&= R ^\nabla_{X,Y} + t [\nabla _X , A_Y ] - t [\nabla _Y , A_X ] + t^2 [A_X, A_Y]\\
&=  R ^\nabla_{X,Y} + t \nabla _X (A_Y) - t \nabla _Y (A_X) + t^2 [A, A]_{X,Y}\\
&=  R ^\nabla_{X,Y} + t (d^\nabla A)_{X,Y} + t^2 [A, A]_{X,Y}\\
\end{aligned}
$$
Thus
$$
\aligned F(\frac 12|| R^{\nabla^t} ||^2)=F(\frac 12||R ^\nabla
||^2 +t \langle R ^\nabla ,d^\nabla A \rangle + \varepsilon(t^2))
\endaligned
$$
where $\varepsilon (t^2) = o (t^2)\quad \text{as}\, t \to 0\, . $
Therefore
$$
\mathcal{YM}_F(\nabla ^t)=\int _M F(\frac 12||R ^\nabla
||^2+t\langle R ^\nabla ,d^\nabla A \rangle+\varepsilon(t^2) )\,  dv
$$
and
$$
\aligned \frac d{dt}\mathcal{YM}_F(\nabla ^t)|_{t=0} &=\int_MF^{\prime
}(\frac
12||R ^\nabla ||^2)\langle R^\nabla ,d^\nabla A \rangle \, dv\\
&=\int_M\langle\delta ^\nabla \big(F^{\prime }(\frac 12||R
^\nabla ||^2)R ^\nabla \big), A \rangle \, dv
\endaligned
$$
This derives the Euler-Lagrange equation for $\mathcal{YM}_F$ as follows
$$
\aligned 0 &=\delta ^\nabla \big(F^{\prime }(\frac 12||R ^\nabla
||^2)R^\nabla \big) \\
&= - \sum_{i=1}^m (\nabla _{e_i}F^{\prime }(\frac 12||R ^\nabla ||^2)R
^\nabla)(e_i,\cdot ) \\
&=F^{\prime }(\frac 12||R ^\nabla ||^2)\delta ^\nabla R
^\nabla -i_{\text{grad}(F^{\prime }(\frac 12||R^\nabla ||^2))}R
^\nabla
\endaligned
$$
\end{proof}

\begin{example}
The Euler-Lagrangian equation for $p$-Yang-Mills functional $\mathcal{YM}_p$ is
\begin{equation}
\delta ^\nabla ( ||R
^\nabla ||^{p-2}R ^\nabla ) =0\tag{3.4}
\end{equation}
\end{example}
 \[\qquad \text{or}\qquad ||R
^\nabla ||^{p-2}\delta ^\nabla R
^\nabla - i_ {\text{grad} (||R
^\nabla ||^{p-2})}R
^\nabla =0
\]
\smallskip

\begin{theorem} Every $F-$Yang-Mills
field $R ^\nabla $ satisfies an $F$-conservation law.
\end{theorem}

\begin{proof} It is known that $R ^\nabla $ satisfies the
Bianchi identity
\begin{equation}
d^\nabla R ^\nabla =0  \tag{3.5}
\end{equation}
Therefore, by Lemma 2.3, Lemma 3.1 and (3.5), we immediately
derive the desired
$$
\operatorname{div} S_{F,R ^\nabla }=0
$$
\end{proof}
\begin{definition}
$\omega \in
A^k(\xi )$ $(k\geq 1)$ is said to satisfy a
\emph {$p-$conservation law} $($$p\geq 2$$)$ if $S_{F,\omega }$ is divergence free for $F(t)=\frac 1p(2t)^{\frac p2}\, ,$ i.e. for any vector field
$X$ on $M$, we have
\begin{equation}
|\omega
|^{p-2}\langle\delta ^\nabla \omega ,i_X\omega \rangle+|\omega
|^{p-2}\langle i_Xd^\nabla \omega ,\omega \rangle
-\langle i_{\text{grad}(|\omega
|^{p-2})}\omega
,i_X\omega \rangle = 0 \tag{3.6}\end{equation}

\end{definition}
As an immediate consequence, one has

\begin{corollary}
Every $p-$Yang-Mills
field $R ^\nabla $ satisfies a $p$-conservation law.
\end{corollary}

The $F$-conservation law is crucial to our subsequent development. $F$-Yang-Mills fields in the cases  $F(t)=\sqrt{1+2t} -1\, $ and
$F(t)=1-\sqrt{1-2t}\, $ will be explored in section 8.
\section{Monotonicity Formulae}

In this section, we will establish monotonicity formulae on
Cartan-Hadamard manifolds or more generally on complete manifolds with a pole. We
recall a  Cartan-Hadamard manifold is a complete simply-connected Riemannian
manifold of nonpositive sectional curvature.
A {\it pole} is a point $x_0\in M$ such that the exponential map from the tangent space to $M$ at $x_0$ into $M$ is a diffeomorphism. By the
{\it radial curvature} $K $ of a manifold with a pole, we mean the
restriction of the sectional curvature function to all the planes which contain the unit vector $\partial (x)$ in $T_{x}M$ tangent to the unique geodesic joining $%
x_{0}$ to $x$ and pointing away from $x_{0}.$ Let the tensor $g - dr \bigotimes dr = 0$  on the radial direction $\partial $, and is just the
metric tensor $g$ on the orthogonal complement $\partial ^{\bot}$.
We'll use the following comparison theorems in Riemannian geometry:

\begin{lemma} $($cf. $[$$GW$$])$ Let $(M,g)$ be a complete
Riemannian manifold with a pole $x_0$. Denote by $K_r$ the radial
curvature $K_r$ of $M$.

$($i$)$  If $-\alpha ^2\leq K_r\leq -\beta ^2$ with $\alpha >0$, $\beta
>0$, then
$$
\beta \coth (\beta r)[g-dr\otimes dr]\leq Hess(r)\leq \alpha \coth
(\alpha r)[g-dr\otimes dr]
$$

$($ii$)$  If $K_r = 0$, then
$$
\frac 1r[g-dr\otimes dr] = Hess(r)
$$

$($iii$)$  If $-\frac A{(1+r^2)^{1+\epsilon}}\leq K_r\leq \frac B{(1+r^2)^{1+\epsilon}}$ with $\epsilon > 0\, , A \ge 0\, ,$ and $0 \le B < 2\epsilon\, ,$
then

$$
\frac{1-\frac B{2\epsilon}}{r}[g-dr\otimes dr]\leq Hess(r)\leq \frac{e^{\frac {A}{2\epsilon}}}{r}[g-dr\otimes dr]
$$

$($iv$)$  If $-Ar^{2q}\leq K_r\leq -Br^{2q}$ with $A\geq B>0$ and
$q>0$, then
$$
B_0r^q[g-dr\otimes dr]\leq Hess(r)\leq (\sqrt{A}\coth \sqrt{A}%
)r^q[g-dr\otimes dr]
$$

for $r\geq 1$, where $B_0=\min
\{1,-\frac{q+1}2+(B+(\frac{q+1}2)^2)^{1/2}\}$.
\end{lemma}

\begin{proof}$(i)\, ,$ $(ii)\, ,$ and $(iv)\, $ are treated in section 2 of [GW].

\noindent
$(iii)$ Since for every $\epsilon >0\, ,$ $$\frac {d }{ds} \big (-\frac 1{2\epsilon} (1+s^2)^{-\epsilon}\big )=  \frac s{(1 + s^2)^{1+\epsilon}}\, , $$
we have
$$\int _0^{\infty} s \frac A{(1 + s^2)^{1+\epsilon}}ds = \frac A{2\epsilon} < \infty \quad \operatorname{and}\quad \int _0^{\infty} s \frac B{(1 + s^2)^{1+\epsilon}}ds = \frac B{2\epsilon} < 1\, .$$
Now the assertion is an immediate consequence of Quasi-isometry Theorem due to Greene-Wu [GW, p.57] in which $1 \le \eta \le e^{\frac A{2\epsilon}}$ and $1-\frac B{2\epsilon} \le \mu \le 1\, .$\end{proof}
In analogous to [Ka], (in which $(iv)$ is employed) for a given function $F$, we introduce the following
\begin{definition}  The $F$-degree $d_F$ is defined to be

$$
d_F=\sup_{t\geq 0}\frac{tF^{\prime }(t)}{F(t)}
$$
\end{definition}
For the most part of this paper, $d_F$ is assumed to be finite,
unless otherwise stated.

\begin{lemma} Let $M$ be a complete manifold with a pole $x_0$.
Assume that there exist two positive functions $h_1(r)$ and
$h_2(r)$ such that
\begin{equation}
h_1(r)[g-dr\otimes dr]\leq Hess(r)\leq h_2(r)[g-dr\otimes dr]
\tag{4.1}
\end{equation}
on $M\backslash \{x_0\}$. If $h_2(r)$ satisfies
\begin{equation}
rh_2(r)\geq 1  \tag{4.2}
\end{equation}
Then
\begin{equation}
\langle S_{F,\omega },\nabla \theta _X \rangle\,\geq
\,\big(1+(m-1)rh_1(r)-2pd_F  r h_2(r)\big)F( \frac{|\omega |^2}2)  \tag{4.3}
\end{equation}
where $X=r \nabla  r$.
\end{lemma}

\begin{proof}Choosing an orthonormal frame $\{e_i,\frac
\partial
{\partial r}\}_{i=1,...,m-1}$ around $x\in M\backslash \{x_0\}$.
Take $ X=r\nabla r$. Then
\begin{equation}
\nabla _{\frac \partial {\partial r}}X=\frac \partial {\partial r}
\tag{4.4}
\end{equation}
\begin{equation}
\nabla _{e_i}X=r\nabla _{e_i}\frac \partial {\partial
r}=rHess(r)(e_i,e_j)e_j \tag{4.5}
\end{equation}
Using (2.6), (2.9), (4.4) and (4.5), we have
\begin{equation}
\aligned
\langle S_{F,\omega },\nabla \theta _X \rangle&=F(\frac{|\omega |^2}2)(1+%
\sum_{i=1}^{m-1}rHess(r)(e_i,e_i)) \\
&\qquad -\sum_{i,j=1}^{m-1}F^{\prime }(\frac{|\omega |^2}2)(\omega \odot
\omega
)(e_i,e_j)rHess(r)(e_i,e_j) \\
&\qquad -F^{\prime }(\frac{|\omega |^2}2)(\omega \odot \omega )(\frac
\partial {\partial r},\frac \partial {\partial r})
\endaligned \tag{4.6}
\end{equation}
By (4.1), we get
\begin{equation}
\aligned \langle S_{F,\omega },\nabla \theta _X \rangle &\geq F(\frac{|\omega
|^2}2)\big(1+(m-1)rh_1(r)\big)
\\
&\qquad -F^{\prime }(\frac{|\omega |^2}2)\sum_{i=1}^{m-1}(\omega \odot
\omega)(e_i,e_i)rh_2(r) \\
&\qquad -F^{\prime }(\frac{|\omega |^2}2)(\omega \odot \omega )(\frac
\partial{\partial r},\frac \partial {\partial r}) \\
&\geq F(\frac{|\omega |^2}2)\big(1+(m-1)rh_1(r)-2pd_Frh_2(r)\big) \\
&\qquad +F^{\prime }(\frac{|\omega |^2}2)(rh_2(r)-1)\langle i_{\frac\partial{\partial r}}\omega, i_{\frac\partial{\partial r}}\omega \rangle
\endaligned \tag{4.7}
\end{equation}
The last step follows from the fact that
$$
\aligned & \sum_{i=1}^{m-1}(\omega \odot
\omega)(e_i,e_i)+(\omega \odot \omega )(\frac
\partial{\partial r},\frac \partial {\partial r})\\ &\quad = \sum_{1\le j_1<\cdots
<j_{p-1}\le m}\sum_{i=1}^{m}\< \omega (e_i,e_{j_1}, \cdots, e_{j_{p-1}}), \omega (e_i,e_{j_1}, \cdots, e_{j_{p-1}})\>
\\
&\quad \le p |\omega|^2\, ,
\endaligned
$$
where $e_m = \frac
\partial{\partial r}\, .$
Now the Lemma follows immediately from (4.2) and (4.7).
\end{proof}

\begin{theorem} Let $(M,g)$ be an $m-$dimensional complete
Riemannian manifold with a pole $ x_0$. Let $\xi :E\rightarrow M$
be a Riemannian vector bundle on $M$ and $ \omega \in A^p(\xi )$.
Assume that the radial curvature $K_r$ of $M$ satisfies one of the
following three conditions:

$($i$)$ $-\alpha ^2\leq K_r\leq -\beta ^2$ with $\alpha >0$, $\beta
>0$ and $ (m-1)\beta -2p\alpha d_F\geq 0$;

$($ii$)$ $K_r = 0$ with $m-2pd_F>0$;

$($iii$)$  $-\frac A{(1+r^2)^{1+\epsilon}}\leq K_r\leq \frac B{(1+r^2)^{1+\epsilon}}$ with $\epsilon > 0\, , A \ge 0\, ,$ $0 < B < 2\epsilon\, ,$ and

$\qquad m - (m-1)\frac B{2\epsilon} -2p e^{\frac {A}{2\epsilon}}d_F > 0$.

If $\omega $ satisfies an $F-$conservation law, then
\begin{equation}
\frac 1{\rho _1^\lambda }\int_{B_{\rho _1}(x_0)}F(\frac{|\omega
|^2}2) dv \leq \frac 1{\rho _2^\lambda }\int_{B_{\rho
_2}(x_0)}F(\frac{|\omega |^2}2) dv\tag{4.8}
\end{equation}
for any $0<\rho _1\leq \rho _2$, where
\begin{equation}
\lambda =\left\{
\begin{array}{cc}
m-2p\frac \alpha \beta d_F &\text {if } K_r  \text { satisfies $($i$)$}\\
m-2pd_F &\text {if } K_r \text{ satisfies $($ii$)$}\\
m - (m-1)\frac B{2\epsilon} -2p e^{\frac {A}{2\epsilon}}d_F &\text{if } K_r \text{
satisfies $($iii$)$}\, .
\end{array}
\right.  \tag{4.9}
\end{equation}
\end{theorem}

\begin{proof}Take a smooth vector field $X=r\nabla r$ on $M\, .$ If $K_r$ satisfies $($i$)$, then by Lemma 4.1 and the increasing function $\beta r \coth (\beta r)\to 1\, $ as $r \to 0\, ,$ (4.2) holds. Now Lemma 4.2 is applicable and by (4.3), we have on $B_\rho(x_0)\backslash\{x_0\}\, ,$ for every $\rho >0,$
$$
\aligned \langle S_{F,\omega },\nabla \theta _X \rangle\,& \geq
\,\big(1+(m-1)\beta r \coth (\beta r)-2pd_F  \alpha r \coth
(\alpha r)\big)F( \frac{|\omega |^2}2) \\ &=\,\big(1+ \beta r \coth (\beta r)(m-1-2pd_F  \frac{\alpha r \coth
(\alpha r)}{\beta r \coth (\beta r)})\big)F( \frac{|\omega |^2}2) \\
& > \,\big(1+ 1 \cdot (m-1-2pd_F  \cdot \frac{\alpha }{\beta } \cdot 1)\big)F( \frac{|\omega |^2}2) = \lambda F( \frac{|\omega |^2}2)\, ,
\endaligned
$$
provided that $m-1-2pd_F  \cdot \frac{\alpha }{\beta }  \ge 0\, ,$ since $\beta r \coth
(\beta r) > 1\, $ for $r > 0\, ,$  and $\frac{\coth
(\alpha r)}{\coth (\beta r)} < 1\, ,$ for $0 < \beta < \alpha\, ,$ and $\coth$ is a decreasing function. Similarly, from Lemma 4.1 and Lemma 4.2, the above inequality holds for the cases (ii), and (iii) on $B_\rho(x_0)\backslash\{x_0\}\, .$  Thus, by the continuity of $\langle S_{F,\omega },\nabla \theta _X \rangle\, $ and $ F( \frac{|\omega |^2}2)\, ,$ and (2.6), we
have for every $\rho >0,$
\begin{equation}
\aligned
&\langle S_{F,\omega },\nabla \theta _X\rangle\geq \lambda
F(\frac{|\omega |^2}2)\qquad \text{in} \quad B_\rho(x_0) \\
& \rho\, \, F(\frac{|\omega |^2}2)\geq S_{F,\omega }(X,\frac \partial {\partial r}) \qquad \text{on} \quad \partial B_\rho(x_0) \endaligned\tag{4.10}
\end{equation}
It follows from (2.11) and (4.10) that
\begin{equation}
\rho\int_{\partial B_\rho(x_0)}F(\frac{|\omega |^2}2) ds \geq \lambda
\int_{B_\rho(x_0)}F( \frac{|\omega |^2}2) dv \tag{4.11}
\end{equation}
Hence we get from (4.11) the following
\begin{equation}
\frac{\int_{\partial B_\rho(x_0)}F(\frac{|\omega |^2}2) ds}{\int_{B_\rho(x_0)}F(\frac{%
|\omega |^2}2) dv} \geq \frac \lambda \rho  \tag{4.12}
\end{equation}
The coarea formula implies that
$$
\frac d{d\rho}\int_{B_\rho(x_0)}F(\frac{|\omega |^2}2) dv =\int_{\partial B_\rho(x_0)}F(%
\frac{|\omega |^2}2) ds
$$
Thus we have
\begin{equation}
\frac{\frac d{d\rho}\int_{B_\rho(x_0)}F(\frac{|\omega
|^2}2) dv}{\int_{B_\rho(x_0)}F( \frac{|\omega |^2}2) dv}\geq \frac \lambda
\rho
\tag{4.13}
\end{equation}
for a.e. $\rho >0\, .$ By integration (4.13) over $[\rho _1,\rho _2]$, we have
$$
\ln \int_{B_{\rho _2}(x_0)}F(\frac{|\omega |^2}2) dv -\ln
\int_{B_{\rho _1}(x_0)}F(\frac{|\omega |^2}2) dv \geq \ln \rho
_2^\lambda -\ln \rho _1^\lambda $$
This proves (4.8).
\end{proof}

\begin{remark} (a) The Theorem is obviously trivial when
$\lambda \leq 0$. (b) A study of Laplacian
comparison  on Cartan-Hadmard manifolds with $Ric_M\leq -\beta ^2$ has been made in [Di].
By employing our techniques, as in the proofs of Lemma 4.2 and Theorem 4.1, some monotonicity formulas under  appropriate curvature
conditions, can be derived. (c) Whereas curvature assumptions $(i)$ to $(iii)$ cannot be exhaustive, our method is unified in the following sense:  Regardless how radial curvature varies, as long as we have Hessian comparison estimates $(4.1)$ with bounds satisfying $(4.2)\, ,$ and the factor $1+(m-1)rh_1(r)-2pd_F  r h_2(r) \ge  c > 0$ in $(4.3)$ for some constant $c$, and $\omega$ satisfies an $F-$ conservation law, then we obtain a monotonicity formula $(4.8)$ for $\mathcal{E}_{F,g}(\omega )-$energy, for an appropriate $\lambda > 0\, .$
\end{remark}

\begin{corollary} Suppose $M$ has constant sectional
curvature $-\alpha ^2$ $($$\alpha ^2\geq 0$$)$. Let $m- 1 - 2pd_F \ge 0$, if $\alpha \ne 0\, ,$ and  $m - 2pd_F > 0$ if $\alpha=0$. Let
$\omega \in A^p(\xi )$ be a $\xi -$valued $p-$form on $M^m$
satisfying an $F$-conservation law. Then
$$
\frac 1{\rho _1^{m-2pd_F}}\int_{B_{\rho _1}(x_0)}F(\frac{|\omega
|^2}2) \, dv \leq \frac 1{\rho _2^{m-2pd_F}} \int_{B_{\rho
_2}(x_0)}F(\frac{|\omega |^2}2)\,  dv
$$
for any $x_0\in M$ and $0<\rho _1\leq \rho _2$.
\end{corollary}

\begin{proof} In Theorem 4.1, if we take $\alpha =\beta \ne 0$ for the
case (i) or $a=0$ for the case (ii), this corollary follows from
(4.8) immediately. \end{proof}

\begin{remark} When $F(t)=t$ and $\omega $ is the
differential of a harmonic map or the curvature form of a
Yang-Mills connection, then we recover the well-known monotonicity
formulae for the harmonic map or Yang-Mills field (cf. [PS]).
\end{remark}

\begin{proposition} Let $(M,g)$ be an $m-$dimensional
complete Riemannian manifold whose radial curvature satisfies

$($iv$)$ $-Ar^{2q}\leq K_r\leq -Br^{2q}$ with $A\geq B>0$ and $q>0$.

Let $\omega \in A^p(\xi )$ satisfy an $F-$conservation law, and $\delta
:=(m-1)B_0-2pd_F\sqrt{A}\coth \sqrt{A} \geq 0$, where $B_0$ is
given in Lemma 4.1. Suppose $(4.15)$ holds. Then
\begin{equation}
\frac 1{\rho _1^{1 + \delta}}\int_{B_{\rho
_1}(x_0)-B_1(x_0)}F(\frac{|\omega |^2}2) \, dv \leq \frac 1{\rho
_2^{1 + \delta}}\int_{B_{\rho
_2}(x_0)-B_1(x_0)}F(\frac{|\omega |^2}2) \, dv\tag{4.14}
\end{equation}
for any $1\leq \rho _1\leq \rho _2$.
\end{proposition}

\begin{proof} Take $X=r\nabla r$. Applying Lemma 4.1, (4.2), and (4.3), we have
$$\aligned \langle S_{F,\omega },\nabla \theta _X \rangle &\geq F(\frac{|\omega
|^2}2)\big(1+\delta r^{q+1}\big)
\endaligned
$$
and
$$\aligned & S_{F,\omega }(X,\frac \partial {\partial r}) = F(\frac{|\omega |^2}2) - F^{\prime }(\frac{|\omega |^2}2)\langle i_{\frac\partial{\partial r}}\omega, i_{\frac\partial{\partial r}}\omega \rangle\qquad \text{on} \quad \partial B_1(x_0) \\
&  S_{F,\omega }(X,\frac \partial {\partial r}) = \rho F(\frac{|\omega |^2}2) - \rho F^{\prime }(\frac{|\omega |^2}2)\langle i_{\frac\partial{\partial r}}\omega, i_{\frac\partial{\partial r}}\omega \rangle \qquad \text{on} \quad \partial B_\rho(x_0) \endaligned
$$

It follows from (2.11) that

\begin{equation*} \aligned & \rho\int_{\partial B_\rho(x_0)} F(\frac{|\omega |^2}2) - F^{\prime }(\frac{|\omega |^2}2)\langle i_{\frac\partial{\partial r}}\omega, i_{\frac\partial{\partial r}}\omega \rangle\,  ds - \int_{\partial B_1(x_0)} F(\frac{|\omega |^2}2) - F^{\prime }(\frac{|\omega |^2}2)\langle i_{\frac\partial{\partial r}}\omega, i_{\frac\partial{\partial r}}\omega \rangle\,  ds\\
& \ge  \int_{B_\rho(x_0)-B_1(x_0)} (1+\delta
r^{q+1}) F(\frac{|\omega |^2}2)\, .
\endaligned
\end{equation*}
Whence, if
\begin{equation}
\qquad \int_{\partial B_1(x_0)} F(\frac{|\omega |^2}2) - F^{\prime }(\frac{|\omega |^2}2)\langle i_{\frac\partial{\partial r}}\omega, i_{\frac\partial{\partial r}}\omega \rangle\,  ds \ge 0\, ,\tag{4.15}
\end{equation}
then
$$\rho\int_{\partial B_\rho(x_0)}F(\frac{|\omega |^2}2) \, ds \geq (1+\delta
)\int_{B_\rho(x_0)-B_1(x_0)}F(\frac{|\omega |^2}2) \, dv $$
for any $\rho > 1\, .$
Coarea formula then implies
\begin{equation}
\frac{d\int_{B_\rho(x_0)-B_1(x_0)}F(\frac{|\omega |^2}2) \, dv}{
\int_{B_\rho(x_0)-B_1(x_0)}F(\frac{|\omega |^2}2) \, dv}\geq \frac{1+\delta} \rho d\rho
\tag{4.16}
\end{equation}
for a.e. $\rho\geq 1$. Integrating (4.16) over $[\rho _1,\rho _2]$, we
get
\[\aligned &\ln \big(\int_{B_{\rho _2}(x_0)-B_1(x_0)}F(\frac{|\omega
|^2}2)\, dv \big)-\ln
\big( \int_{B_{\rho _1}(x_0)-B_1(x_0)}F(\frac{|\omega |^2}2)\, dv \big)  \\
& \quad \geq  (1+\delta) \ln \rho _2- (1+\delta) \ln \rho _1 \endaligned
\]
Hence we prove the proposition.
\end{proof}

\begin{corollary} Let $K_r\, $ and $\delta$ be as in Proposition 4.1, and $\, \omega $ satisfy an $F-$conservation law. Suppose
\begin{equation}d_F\, \big |i_{\frac {\partial}{\partial r}}\omega\big |^2 \le \frac {|\omega|^2}{2} \tag{4.17}
\end{equation}
on $\partial B_1\, ,$ or $F(\frac{|\omega |^2}2) - F^{\prime }(\frac{|\omega |^2}2)|i_{\frac\partial{\partial r}}\omega|^2 \ge 0 $ on $\, \partial B_1\, .$
Then $(4.14)$ holds.
\end{corollary}
\begin{proof}
The assumption (4.17) implies that $(4.15)$ holds, and the assertion follows from Proposition 4.1.
\end{proof}

\section{Vanishing Theorems and Liouville Type Results}
In this section we list some results in the following three subsections, that are immediate applications of the monotonicity formulae in the last section.

\subsection{Vanishing theorems for vector bundle valued $p$-forms}

\begin{theorem} Suppose the radial curvature $K_r$ of $M$ satisfies the condition
in Theorem 4.1. If $\omega \in A^p(\xi )$ satisfies an
$F-$conservation law and
\begin{equation}
\int_{B_\rho(x_0)}F(\frac{|\omega |^2}2)\, dv = o(\rho^\lambda )\quad \text{as
} \rho\rightarrow \infty\tag{5.1}
\end{equation}
where $\lambda $ is given by (4.9), then $F(\frac{|\omega |^2}2)\equiv 0\, ,$ and hence $\omega \equiv 0$. In
particular, if $\omega $ has finite $\mathcal{E}_{F,g}-$energy, then
$\omega \equiv 0$.
\end{theorem}

\begin{definition}
$\omega \in A^p(\xi)$ is said to have \emph{slowly divergent $\mathcal{E}_{F,g}-$energy}, if there exists a positive continuous function $\psi (r)$ such that
\begin{equation}
\int_{\rho_1}^\infty \frac{dr}{r\psi (r)}=+\infty\tag{5.2}
\end{equation}
for some $\rho_1>0\, ,$ and
\begin{equation}
\lim_{\rho\rightarrow \infty }\int_{B_\rho(x_0)}\frac{F(\frac{|\omega
|^2}2)}{\psi (r(x))} \, dv <\infty\tag{5.3}
\end{equation}
\end{definition}

\begin{remark} (1) Hesheng Hu introduced the notion of slowly divergent energy (in which $F(t)=t\, ,$ $\omega = du\, ,$ or
$\omega =R^\nabla \, $), and made a pioneering study in [Hu1,2]. (2) In [LL2]
and [LSC], the authors established some Liouville results for
$F-$harmonic maps or forms with values in a vector bundle satisfying
an $F-$conservation law under the condition of slowly divergent
energy. Obviously Theorem 5.1 improves all these growth conditions, as its special cases of $F\, ,$
and expresses the growth condition more explicitly (cf. Theorem 10.1, Examples 10.1 and 10.2 in Appendix).
\end{remark}

\begin{theorem} Suppose $M$ and $\delta$ satisfy the condition
in Proposition 4.1. If $\omega \in A^p(\xi )$ satisfies an
$F-$conservation law, $(4.15)$ holds, and
\begin{equation}
\int_{B_\rho (x_0)}F(\frac{|\omega |^2}2) dv =o(\rho^{1 + \delta})\text{\quad as }\rho\rightarrow \infty\tag{5.4}
\end{equation}
then $\omega \equiv 0$ on $M-B_1(x_0)$. In particular, if $\omega
$ has finite $\mathcal{E}_{F,g}-$energy, then $\omega \equiv 0$ on
$M-B_1(x_0)$.
\end{theorem}

Notice that Theorem 5.2 only asserts that $\omega $ vanishes in
an open set of $M$. If $\omega $ posses the unique continuation
property, then $\omega $ vanishes on $M$ everywhere (cf. Corollaries 5.2 and 5.4).

\subsection{Liouville theorems for $F$-harmonic maps}

Let $u: M \to N$ be an $F-$harmonic map. Then its differential $du $ can be viewed as a $1$-form with values in the induced bundle $u^{-1}TN\, .$ Since
$\omega=du $ satisfies an $F-$conservation law, we obtain
the following Liouville-type

\begin{theorem} Let $N$ be a Riemannian manifold. Suppose the radial curvature $K_r$ of $M$ and  $\lambda $ satisfy the condition
in Theorem 4.1 in which $p=1\, .$  Then every $F-$harmonic map $u: M \to N$ with the following growth
condition is a constant.
\begin{equation}
\int_{B_\rho(x_0)}F(\frac{|du |^2}2)\, dv = o(\rho^\lambda )\quad \text{as
} \rho\rightarrow \infty\tag{5.5}
\end{equation}
 In particular, every $F-$harmonic map $u: M \to N$ with finite $F-$energy is a constant.
\end{theorem}

\begin{proof} This follows at once from Theorem 5.1 in which $p=1$ and $\omega = du\, .$
\end{proof}

\begin{remark} This is in contrast to a Liouville Theorem for $F$-harmonic maps into a domain of strictly convex function by a different approach (cf. Theorem 12.1 in [We2]).
\end{remark}

\begin{theorem}[Liouville Theorem for $p$-harmonic maps] Let $N$ be a Riemannian manifold. Suppose the radial curvature $K_r$ of $M$ satisfies one of the following three conditions:

$($i$)$ $-\alpha ^2\leq K_r\leq -\beta ^2$ with $\alpha >0$, $\beta
>0$ and $ (m-1)\beta -p\alpha\geq 0$;

$($ii$)$ $K_r = 0$ with $m - p
>0$;

$($iii$)$  $-\frac A{(1+r^2)^{1+\epsilon}}\leq K_r\leq \frac B{(1+r^2)^{1+\epsilon}}$ with $\epsilon > 0\, , A \ge 0\, ,$ $0 < B < 2\epsilon\, ,$ and

$\qquad m - (m-1)\frac B{2\epsilon} - p e^{\frac {A}{2\epsilon}} > 0$.

Then every $p-$harmonic map $u: M \to N$ with the following $p-$energy growth
condition (5.6) is a constant.
\begin{equation}
\frac 1p \int_{B_\rho(x_0)}|du|^p\, dv = o(\rho^\lambda )\quad \text{as
} \rho\rightarrow \infty \tag{5.6}
\end{equation}
where \begin{equation}
\lambda =\left\{
\begin{array}{cc}
m-p\frac \alpha \beta, &\text {if } K_r  \text { satisfies $($i$)$}\\
m-p, &\text {if } K_r \text{ satisfies $($ii$)$}\\
m - (m-1)\frac B{2\epsilon} - p e^{\frac {A}{2\epsilon}}, &\text {if } K_r \text{
satisfies $($iii$)$}.\end{array}
\right.  \tag{5.7}
\end{equation}
In particular,  every $p-$harmonic map $u: M \to N$ with finite $p-$energy is a constant.
\end{theorem}

\begin{proof} This follows immediately from Theorem 5.3 in which $F(t)=\frac 1p(2t)^{\frac p2}\, $ and $d_F = \frac p2\, .$
\end{proof}

\begin{remark}
The case $\frac 1p \int_{B_\rho(x_0)}|du|^p\, dv = o((\ln \rho)^q)\quad \text{as
} \rho\rightarrow \infty$ for some positive number $q$ is due to Liu-Liao [LL1].
\end{remark}

\begin{corollary} Let $M\, ,$ $N\, ,$ $K_r\, ,$ $\lambda\, $ and the growth
condition (5.6) be as in Theorem 5.4, in which $p=2\, .$ Then every harmonic map $u: M \to N$ is a constant.
\end{corollary}

\begin{theorem} Let $M\, ,$ $N\, ,$ $K_r\, ,$ and $\delta\, $ satisfy the condition
of Proposition 4.1 in which $p=1\, .$ Suppose $(4.15)$ holds for $\omega = du$.
Then every $F-$harmonic map $u: M \to N$ with the following growth
condition is a constant on $M-B_1(x_0)$:
\begin{equation}
\int_{B_\rho(x_0)}F(\frac{|du|^2}2) dv =o(\rho^{1 + \delta})\text{\quad as }\rho\rightarrow \infty\tag{5.8}
\end{equation}
on $M-B_1(x_0)\, .$ In particular, if $u
$ has finite $F-$energy, then $u \equiv \operatorname{const}$ on
$M-B_1(x_0)$.
\end{theorem}

\begin{proof} This follows at once from Proposition 4.1.
\end{proof}

\begin{proposition}  Let $(M,g)$ be an $m-$dimensional
complete Riemannian manifold whose radial curvature satisfies
$-Ar^{2q}\leq K_r\leq -Br^{2q}\,
$
with $A\geq B>0$ and $q>0$. If $\delta
:=(m-1)B_0-p\sqrt{A}\coth \sqrt{A} \geq 0$, where $B_0$ is
given in Lemma 4.1. Suppose $(4.15)$ holds for $\omega = du$.  Then every $p-$harmonic map $u: M \to N$ with the growth
condition $\frac 1p \int_{B_\rho (x_0)}|du|^p\, dv = o(\rho^{1 + \delta})\text{\quad as }\rho\rightarrow \infty$ is a constant on $M-B_1(x_0)\, ,$
In particular, if $u
$ has finite $p-$energy, then $u \equiv \operatorname{const}$ on
$M-B_1(x_0)$.
\end{proposition}

\begin{corollary} Let $M\, ,$ $N\, ,$ $K_r\, ,$ $\delta\, ,$ $(4.15)\, ,$ and the growth condition
be as in Proposition 5.1, in which $p=2\, .$  Then every harmonic map $u: M \to N$ is a constant.
\end{corollary}

\begin{proof} This follows immediately from Proposition 5.1 and the unique continuation property of a harmonic map.
\end{proof}

\subsection{Applications in $F$-Yang-Mills fields}

Let $R^\nabla $ be an $F-$Yang-Mills field, associated with an $F-$Yang-Mills connection $\nabla$ on the adjoint bundle $Ad(P)$ of a principle $G$-bundle over a manifold $M\, .$ Then $R^\nabla $ can be viewed as a $2$-form with values in the adjoint bundle over $M\, ,$ and by Theorem 3.1,
$\omega = R^\nabla$ satisfies an $F-$conservation law.

 \begin{theorem}[Vanishing Theorem for $F$-Yang-Mills fields]Let $M\, ,$ $K_r\, ,$ and $\lambda $ satisfy the condition
in Theorem 4.1 in which $p=2\, .$ Suppose $F-$Yang-Mills field  $R^{\nabla}$ satisfies the following growth
condition
\begin{equation}
\int_{B_\rho(x_0)} F(\frac{|R^{\nabla} |^2}2)\, dv = o(\rho^\lambda )\quad \text{as
} \rho\rightarrow \infty.\tag{5.9}
\end{equation}
Then $R^{\nabla} \equiv 0$ on $M\, .$
 In particular, every $F-$Yang-Mills field $R^{\nabla}$  with finite $F-$Yang-Mills energy vanishes on $M$.
\end{theorem}

\begin{proof} This follows at once from Theorem 5.1 in which $p=2$ and $\omega = R^\nabla\, .$
\end{proof}

\begin{theorem}[Vanishing Theorem for $p$-Yang-Mills fields] Suppose the radial curvature $K_r$ of $M$ satisfies the one of the following conditions:

$($i$)$ $-\alpha ^2\leq K_r\leq -\beta ^2$ with $\alpha >0$, $\beta
>0$ and $ (m-1)\beta -2p\alpha\geq 0$;

$($ii$)$ $K_r = 0$ with $m - 2p >0$;

$($iii$)$  $-\frac A{(1+r^2)^{1+\epsilon}}\leq K_r\leq \frac B{(1+r^2)^{1+\epsilon}}$ with $\epsilon > 0\, , A \ge 0\, ,$ and $0 < B < 2\epsilon\, ,$
and

$\qquad m - (m-1)\frac B{2\epsilon} -2p e^{\frac {A}{2\epsilon}} > 0$.

Then every $p-$Yang-Mills field $R^{\nabla}$ with the following growth
condition vanishes:
\begin{equation}
\frac 1p \int_{B_\rho(x_0)}|R^{\nabla } |^p \, dv = o(\rho^\lambda )\quad \text{as
} \rho\rightarrow \infty\tag{5.10}
\end{equation}
where \begin{equation}
\lambda =\left\{
\begin{array}{cc}
m-2p\frac \alpha \beta, &\text {if } K_r  \text { satisfies $($i$)$}\\
m-2p, &\text {if } K_r \text{ satisfies $($ii$)$}\\
m - (m-1)\frac B{2\epsilon} -2p e^{\frac {A}{2\epsilon}} > 0, &\text {if } K_r \text{
satisfies $($iii$)$}\, .\end{array}
\right.  \tag{5.11}
\end{equation}
In particular, every $p-$Yang-Mills field $R^{\nabla}$  with finite $\mathcal{YM}_p-$energy vanishes on $M$.
\end{theorem}

\begin{corollary} Let  $M\, ,$ $N\, ,$ $K_r\, ,$ $\lambda\, ,$ and the growth
condition (5.12) be as in Theorem 5.10, in which $p=2\, .$ Then every Yang-Mills field $R^\nabla \equiv 0$ on $M$.
\end{corollary}
\begin{theorem} Suppose $M\, ,$  $K_r\, ,$ and $\delta\, ,$  satisfy the same conditions
of Proposition 4.1 in which $p=2\, ,$ and $(4.15)$ holds for $\omega=R^{\nabla}\, .$ Then every $F-$Yang-Mills field $R^\nabla$ with the following growth
condition vanishes on $M-B_1(x_0)$:
\begin{equation}
\int_{B_\rho(x_0)}F(\frac{|R^\nabla|^2}2) dv =o(\rho^{1 + \delta
})\text{\quad as }\rho\rightarrow \infty\tag{5.12}
\end{equation}
In particular, if $R^\nabla$ has finite $F-$Yang-Mills energy, then $R^\nabla \equiv 0$ on
$M-B_1(x_0)$.
\end{theorem}

\begin{proof} This follows immediately from Proposition 4.1.
\end{proof}

\begin{proposition}  Let $(M,g)$ be an $m-$dimensional
complete Riemannian manifold whose radial curvature satisfies
$-Ar^{2q}\leq K_r\leq -Br^{2q}\,
$
with $A\geq B>0$ and $q>0$. Let $\delta
:=(m-1)B_0-2p\sqrt{A}\coth \sqrt{A} \geq 0$, where $B_0$ is
given in Lemma 4.1, and let $(4.15)$ hold for $\omega=R^{\nabla}\, .$  Then every $p-$Yang-Mills field $R^\nabla$ with the growth
condition $\frac 1p \int_{B_\rho(x_0)}|R^\nabla|^p\, dv = o(\rho^{1 + \delta
})\text{\quad as }\rho\rightarrow \infty$ vanishes on $M-B_1(x_0)\, ,$
In particular, if $R^\nabla
$ has finite $p-$Yang-Mills energy, then $R^\nabla \equiv 0$ on
$M-B_1(x_0)$.
\end{proposition}

\begin{corollary} Let $M\, ,$ $K_r\, ,$ $\delta\, ,$ $(4.15)\, ,$ and the growth
condition be as in Proposition 5.2, in which $p=2\, .$ Then every Yang-Mills field $R^\nabla \equiv 0$ on $M\, .$
\end{corollary}

\begin{proof} This follows at once from Proposition 5.2, and the unique continuation property of Yang-Mills field.
\end{proof}

Further applications will be treated in Section 8.

\section{Constant Dirichlet Boundary-Value Problems}

To investigate the constant Dirichlet boundary-value problems for $1$-forms, we begin with

\begin{definition}
The \emph{$F$-lower degree} $l_F$ is given by
$$
l_F=\inf_{t\geq 0}\frac{tF^{\prime }(t)}{F(t)}
$$
\end{definition}
\begin{definition}
A bounded domain $D\subset M$ with $C^1$ boundary $\partial D$ is called
\emph{starlike} if there exists an interior point $x_0\in D$ such that
\begin{equation}
\langle\frac \partial {\partial r_{x_0}},\nu\rangle \big{|}_{\partial
D}\geq 0 \tag{6.1}
\end{equation}
where $\nu$ is the unit outer normal to $\partial D\, ,$ and the vector field $\frac \partial {\partial r_{x_0}}$ is the unit
vector field such that for any $x\in D\backslash \{x_0\} \cup \partial D\, ,$ $\frac \partial {\partial r_{x_0}}(x)$ is the unit vector tangent to the unique geodesic joining $%
x_{0}$ to $x$ and pointing away from $x_{0}\, .$
\end{definition}
It is obvious that any convex domain is starlike.
\begin{theorem} Suppose $M$ satisfies the same condition of
Theorem 4.1 and $D\subset M$ is a bounded starlike domain with
$C^1$ boundary. Assume that the $F$-lower degree $ l_F\geq 1/2$. If $\omega \in A^1(\xi
)$ satisfies an $F-$conservation law and annihilates any tangent
vector $\eta $ of $\partial D$, then $\omega $ vanishes on $D$.
\end{theorem}

\begin{proof} By assumption, there exists a point $x_0\in D$
such that the distance function $r_{x_0}$ satisfies (6.1). Take
$X=r\nabla r$, where $r=r_{x_0}$. From the proofs of Theorem 4.1,
we know that
\begin{equation}
\langle S_{F,\omega },\nabla \theta _X\rangle\geq
\, c F(\frac{|\omega |^2}2) \tag{6.2}
\end{equation}
where $c\, $ is a positive constant. Since $\omega \in A^1(\xi )$
annihilates any tangent vector $\eta$ of $\partial D$, we easily
derive the following on $\partial D$
\begin{equation}
\aligned S_{F,\omega }(X,\nu)&=rS_{F,\omega }(\frac
\partial
{\partial r},\nu) \\
&= r\big(F(\frac{|\omega |^2}2)\langle\frac \partial {\partial r},\nu
\rangle - F^{\prime }(\frac{|\omega |^2}2)\langle\omega (\frac
\partial
{\partial r}),\omega (\nu )\rangle\big)\\
&=r\langle\frac \partial {\partial r},\nu
\rangle \big(F(\frac{|\omega |^2}2)-F^{\prime }(\frac{|\omega |^2}2)|\omega |^2\big) \\
&\leq r\langle\frac \partial {\partial r},\nu \rangle
(1-2l_F)F(\frac{|\omega |^2}2)\leq 0
\endaligned \tag{6.3}
\end{equation}
From (2.11), (6.2) and (6.3), we have
$$
0 \le \int_D c F(\frac{|\omega |^2}2) dv \leq 0
$$
which implies that $\omega \equiv 0$.  \end{proof}

\begin{corollary} Suppose $M$ and $D$ satisfy the same
assumptions of Theorem 6.1. Let $u:\overline{D}\rightarrow N$ be
a $p-$harmonic map $($$p\geq 1$$)$ into an arbitrary Riemannian
manifold $N$. If $u|_{\partial D}$ is constant, then $u|_D$ is
constant.
\end{corollary}
\begin{proof} For a $p-$harmonic map $u$, we have $F(t)=\frac 1p
(2t)^{\frac p2}$. Obviously $d_F=l_F=\frac p2$. Take $\omega =du$. This
corollary follows immediately from Theorem 6.1.
\end{proof}

\begin{remark}When $M=\Bbb{R}^m$ and $D=B_\rho(x_0)$, this
result, Corollary 6.1, recaptures the work of Karcher and Wood on the constant Dirichlet boundary-value
problem for harmonic maps [KW]. The result of
Karcher and Wood was also generalized to harmonic maps with
potential by Chen [Ch] and $p-$harmonic maps with potential by Liu
[Li2] for disc domains.
\end{remark}

\section{Extended Born-Infeld fields and Exact Forms}

In this section, we will establish Liouville type theorems for
solutions of the extended Born-Infeld equations (1.5) and (1.6)
proposed by [Ya]. Using Hodge star operator $\ast\, ,$ we can rewrite the
equations (1.5) and (1.6) as
\begin{equation}
d*\left( \frac{d\omega }{\sqrt{1+|d\omega |^2}}\right) =0,\qquad \omega
\in A^p(\mathbb{R}^m)
\tag{7.1}
\end{equation}
and
\begin{equation}
d*\left( \frac{d\sigma }{\sqrt{1-|d\sigma |^2}}\right) =0,\qquad \sigma
\in A^q(\mathbb{R}^m)
\tag{7.2}
\end{equation}
respectively. As pointed out in the introduction, the solutions of
(7.1) and (7.2) are critical points of the $E_{BI}^{+}$-energy functional and the $E
_{BI}^{-}$-energy functional respectively. Notice that the $E_{BI}^{+}$-energy functional and the $E_{BI}^{-}$-energy functional are $E_{F,g}-$energy
functionals with $F(t)=\sqrt{1+2t}-1$ ($t\in [0,+\infty )$), and $F(t)=1-\sqrt{1-2t}$
($t\in [0,1/2)$) respectively.
\begin{definition}
The \emph{extended Born-Infeld energy functional with the plus sign on a manifold} $M$ is the mapping $E_{BI}^{+} : A^p(M)\to \mathbb{R}^+\, $ given by
\begin{equation}
E_{BI}^{+}(\omega )=\int_{M} \sqrt{1+|d\omega |^2}-1\quad dv
\tag{7.3}
\end{equation}
and the \emph{extended Born-Infeld energy functional with the minus sign on a manifold} $M$ is the mapping $E_{BI}^{-} : A^q(M) \to \mathbb{R}^+\, $ given by
\begin{equation}
E_{BI}^{-}(\sigma )=\int_{M} 1-\sqrt{1-|d\sigma |^2}\quad dv
\tag{7.4}
\end{equation}
A critical point $\omega $ of
$E_{BI}^{+}$ $($resp. $\sigma$ of $E_{BI}^{-}$$)$ with respect to any compactly supported variation is called an extended
\emph{Born-Infeld field with the plus sign} $($resp. \emph{with the
minus sign}$)$ on a manifold.
\end{definition}

Obviously Corollary 2.1 implies that the
solutions of (7.1) and (7.2) satisfy $F-$conservation laws.

Now we recall the equivalence between (7.1) and (7.2) found by
[Ya] as follows: Let $\omega \in A^p(\Bbb{R}^m)$ be a solution of
(7.1) with $0\leq p\leq m-2$. Then
\begin{equation}
\tau =\pm *\left(\frac{d\omega }{\sqrt{1+|d\omega |^2}}\right)  \tag{7.5}
\end{equation}
is a closed $(m-p-1)-$form. Since the de Rham cohomology group $H^{m-p-1}(\Bbb{R}^m)=0$, there
exists an $(m-p-2)-$form $\sigma $ such that $\tau =d\sigma $. It
is easy to derive from (7.5) the following
\begin{equation}
|d\omega |^2=\frac{|d\sigma |^2}{1-|d\sigma |^2}  \tag{7.6}
\end{equation}
and
\begin{equation}
d\omega =\pm (-1)^{p(m-p)}*\frac{d\sigma }{\sqrt{1-|d\sigma
|^2}} \tag{7.7}
\end{equation}
The Poincar\'e Lemma implies that $\sigma $ satisfies (7.2) with
$q=m-p-2$. Using (7.6), we get
\begin{equation}
\aligned 1-\sqrt{1-|d\sigma |^2}&=\frac{\sqrt{1+|d\omega
|^2}-1}{\sqrt{1+|d\omega |^2}}\\
&\leq \sqrt{1+|d\omega |^2}-1
\endaligned \tag{7.8}
\end{equation}
Conversely, a solution $\sigma $ of (7.2) with $0\leq q\leq m-2$
gives us a solution $\omega \in A^p(\Bbb{R}^m)$ of (7.1) with
$p=m-q-2$, and
\begin{equation}
\sqrt{1+|d\omega |^2}-1=\frac{1-\sqrt{1-|d\sigma
|^2}}{\sqrt{1-|d\sigma |^2}} \tag{7.9}
\end{equation}

Let's first consider the equation (7.1) and let $\omega $ be a
solution of (7.1). Choose an orthonormal basis $\omega
_1,...,\omega _k$ of $A^p(\Bbb{R}^m)$ consisting of constant
differential forms, where
$
k=\begin{pmatrix}
  m \\
  p \\
\end{pmatrix}\, ,
$ and for each $1 \le \alpha \le k\, ,$
$$
\aligned
\omega _\alpha =dx_{j_1}\wedge \cdots \wedge dx_{j_p}
\endaligned
$$
for some $1\leq j_1<\cdots <j_p\leq m\, .$
Then we may write $\omega
=\sum_{\alpha
=1}^kf^\alpha \omega _\alpha $. So $\omega $ may be regarded as a map $%
\omega :\Bbb{R}^m\rightarrow \Lambda ^p(\Bbb{R}^m)\simeq \Bbb{R}^k$ where $
k=\begin{pmatrix}
  m \\
  p \\
\end{pmatrix}\, .
$ Let $%
M=(x,\omega (x))$ be the graph of $\omega $ in $\mathbb{R}^{m+k}$ and let
$G(\rho)$ be the extrinsic ball of radius $\rho$ of the graph centered
at the origin of $\Bbb{R}^{m+k}$ given by

$$
G(\rho)=M\cap B^{m+k}(\rho)
$$
Set
$$
\omega _\rho=\sum_{\alpha =1}^kf_\rho^\alpha \omega _\alpha
$$
where
\begin{equation}
f_\rho^{\ \alpha }(x)=\left\{
\begin{array}{cc}
\rho\quad &\text{if }f^{\ \alpha }>\rho \\
f^\alpha (x)\quad &\text{if }|f^\alpha (x)|\leq \rho \\
-\rho\quad &\text{if }f^{\ \alpha }<\rho
\end{array}
\right. \tag{7.10}
\end{equation}
For any $\delta >0$, let $\phi $ be a nonnegative cut-off function
defined on $\Bbb{R}^m$ given by
\begin{equation}
\phi =\left\{
\begin{array}{cc}
1 &\text{on }B^m(\rho)\qquad \qquad \\
\frac{(1+\delta )\rho-r(x)}{\delta \rho} &\text{on }B^m((1+\delta
)\rho)\backslash B^n(\rho) \\
0 &\text{on }\mathbb{R}^m\backslash B^m((1+\delta )\rho)
\end{array}
\right. \tag{7.11}
\end{equation}

\begin{proposition} Let $\omega \in A^p(\Bbb{R}^m)$ be an extended
Born-Infeld field with the plus sign on $\Bbb{R}^m$. Then the Born-Infeld type energy of $\omega$ over $G(\rho)$
satisfies the upper bound
$$
E_{BI}^{+}(\omega;G(\rho))\leq m\sqrt{k}\omega _m\rho^m
$$
where $
k=\begin{pmatrix}
  m \\
  p \\
\end{pmatrix}\,
$ and $\omega_m$ is the volume of the unit ball in $\mathbb{R}^m$.
\end{proposition}
\begin{proof} Taking inner product with $\phi \omega _\rho\, ,$ we may get
from (1.5) or (7.1) that
$$
\aligned
0 &=\int_{\Bbb{R}^m} \langle d^{*}(\frac{d\omega }{\sqrt{1+|d\omega |^2}}),\phi \omega _\rho\rangle \, dx \\
&=\int_{\Bbb{R}^m} \langle\frac{d\omega }{\sqrt{1+|d\omega |^2}},d(\phi \omega _\rho)\rangle \, dx \\
&=\int_{\Bbb{R}^m} \langle\frac{d\omega }{\sqrt{1+|d\omega |^2}},d\phi \wedge
\omega _\rho\rangle \, dx + \int_{\Bbb{R}^m} \phi \langle\frac{d\omega
}{\sqrt{1+|d\omega |^2}},d\omega _\rho\rangle \, dx \endaligned
$$
Using the fact that $|d\phi |=|\nabla \phi |  \leq \frac 1{\delta
\rho}$ and $|\omega _\rho|\leq \sqrt{k}\rho$, we have
$$
\aligned \int_{B^m((1+\delta )\rho)}\phi &\langle\frac{d\omega
}{\sqrt{1+|d\omega |^2}} ,d\omega _\rho\rangle \, \, dx \leq
\int_{B^m((1+\delta
)\rho)}\frac{|d\phi ||\omega_\rho||d\omega |}{\sqrt{1+|d\omega |^2}} \, \, dx \\
&\leq \frac{\sqrt{k}}\delta \text{Vol}\big( B^m((1+\delta )\rho)-B^m(\rho))\big)
\endaligned
$$
So
$$
\int_{B^m(\rho)\cap \{|f^\alpha |\leq \rho\}}\frac{|d\omega
|^2}{\sqrt{ 1+|d\omega |^2}} \, dx \leq \frac{\sqrt{k}}\delta
\text{Vol}\big(B^m((1+\delta )\rho)-B^m(\rho))\big)
$$
Because $G(\rho)\subset M\cap (B^m(\rho)\times [-\rho,\rho])$, we have
$$
\aligned E_{BI}^{+}(\omega;G(\rho)) &\leq \int_{B^m(\rho)\cap \{|f^\alpha |\leq
\rho\}}
\sqrt{1+|d\omega |^2}-1 \quad dx \\
&\leq \int_{B^m(\rho)\cap \{|f^\alpha |\leq \rho\}}\sqrt{1+|d\omega |^2}\quad dx
-Vol(B^m(\rho)\cap \{|f^\alpha |\leq \rho\})  \\
&\leq \int_{B^m(\rho)\cap \{|f^\alpha |\leq \rho\}}\frac{|d\omega
|^2+1}{\sqrt{
1+|d\omega |^2}}\quad dx - Vol(B^m(\rho)\cap \{|f^\alpha |\leq \rho\}) \\
&\leq \frac{\sqrt{k}}\delta Vol\big(B^m((1+\delta )\rho)-B^m(\rho))\big)\\
& = \frac{\sqrt{k}}\delta \omega_m \big((1+\delta )^m \rho^m - \rho^m\big)
\endaligned
$$
Let $\delta \rightarrow 0$, we have
$$
E_{BI}^{+}(\omega;G(\rho))\leq m\sqrt{k}\omega _m\rho^m
$$
\end{proof}

\begin{remark} When $\omega =f\in A^0(\Bbb{R}^m)=C^\infty
(\Bbb{R}^m)$, the above result is the volume estimate for the
minimal graph of $f$ (cf. [LW]).
\end{remark}

\begin{lemma}

$($i$)$ If $F(t)=\sqrt{1+2t}-1$ with $t\in
[0,+\infty )$, then $d_F=1$ and $ l_F=1/2$.

$($ii$)$ If $F(t)=1-\sqrt{1-2t}$ with $t\in [0,1/2)$, then
$d_F=+\infty $ and $ l_F=1$.
\end{lemma}

\begin{proof}(i) For $F(t)=\sqrt{1+2t}-1$, we have
\[
\aligned
\frac{tF^{\prime }(t)}{F(t)}&=\frac t{\sqrt{1+2t}(\sqrt{1+2t}-1)} \\
&=\frac{\sqrt{1+2t}+1}{2\sqrt{1+2t}} \quad \operatorname{for}\quad t \in (0, +\infty )
\endaligned
\]
Hence,
\begin{equation}
\aligned
\frac{tF^{\prime }(t)}{F(t)}&= \frac 12 + \frac{1}{2\sqrt{1+2t}} \quad \operatorname{for}\quad t \in [0, +\infty )\, .
\endaligned\tag{7.12}
\end{equation}
By definition, we get $d_F=1$ and $l_F=1/2$.

(ii) For $F(t)=1-\sqrt{1-2t}$, we have
\[
\aligned
\frac{tF^{\prime }(t)}{F(t)}&=\frac t{\sqrt{1-2t}(1-\sqrt{1-2t})} \\
&=\frac{1+\sqrt{1-2t}}{2\sqrt{1-2t}} \quad \operatorname{for}\quad t \in (0, \frac12)\, .
\endaligned
\]
Hence,
\begin{equation}
\aligned
\frac{tF^{\prime }(t)}{F(t)}& =\frac 12 + \frac{1}{2\sqrt{1-2t}}\quad \operatorname{for}\quad t \in [0, \frac12)
\endaligned\tag{7.13}
\end{equation}
By definition, we obtain $d_F=+\infty $ and $l_F=1$.
\end{proof}

By applying Corollary 4.1 to $M=\Bbb{R}^m$ and $F=\sqrt{1+2t}-1$,
we immediately get the following:

\begin{theorem} Let $\omega \in A^p(\Bbb{R}^m)$ be an
extended Born-Infeld field with the plus sign on $\mathbb{R}^m$. If $m>2p$ and $\omega$ satisfies the
following growth condition
$$
\left. \int_{B_\rho(x_0)} \sqrt{1+|d\omega
|^2}-1 \quad dx =o(\rho^{m-2p})\right.\quad \text{as
} \rho\rightarrow \infty
$$
for some point $x_0\in \Bbb{R}^m$, then $d\omega =0\, ,$ and $\omega$ is exact. In
particular, if $\omega $ has finite $E_{BI}^{+}-$energy, then
$\omega$ is exact.
\end{theorem}

\begin{remark} In [SiSiYa], the authors proved the following:
Let $\omega $ be a solution of (7.1). If $d\omega \in
L^2(\Bbb{R}^m)$ ($m\geq 3$) or $d\omega \in L^2(\Bbb{R}^2)\cap
\mathcal{H}^{1}$ on $\Bbb{R}^2$, where $\mathcal{H}^1$ is the Hardy space, then $d\omega \equiv 0$. In view of the
inequality $\sqrt{1+t^2}-1\leq \frac{t^2}2\, $ for any $t \ge 0$, it is clear that
being in $L^2$ ensures finite $E_{BI}^{+}-$energy.
\end{remark}

Using the duality between solutions of (7.1) and (7.2), we have

\begin{proposition} Let $\sigma \in A^q(\Bbb{R}^m)$ be a
$q-$form with $\frac{m-4}2<q< m-2$. If $ \sigma $ is an
extended Born-Infeld field with the minus sign on $\mathbb{R}^m\, ,$ and $\sigma $ satisfies the following growth
\begin{equation}
\int_{B_\rho(x_0)}\frac{1-\sqrt{1-|d\sigma |^2}}{\sqrt{1-|d\sigma
|^2}} \, dx = o(\rho^{2q-m+4})\quad \text{as
} \rho\rightarrow \infty  \tag{7.14}
\end{equation}
then $d\sigma =0\, ,$ and $\sigma$ is exact. In particular, if $\sigma $ has finite $E
_{BI}^{-}-$energy, then $\sigma$ is exact.
\end{proposition}

\begin{proof} By the duality between (7.1) and (7.2), we get a
solution $\omega $ from the solution $\sigma $ of (7.2), where
$\omega \ $ satisfies (7.1) and (7.9). Since $p=m-q-2$, the condition
$q>\frac{m-4}2$ is equivalent to $m>2p$. Obviously (7.9) and
(7.14) imply
$$
\int_{B_\rho(x_0)} \sqrt{1+|d\omega |^2}-1 \quad dx =o(\rho^{m-2p})\quad \text{as
} \rho\rightarrow \infty
$$
Therefore Theorem 7.1 implies that $d\omega =0$ which is
equivalent to $ d\sigma =0$.
\end{proof}

\begin{proposition} Let $\sigma \in A^q(\Bbb{R}^m)$ be a
$q-$form with $q < \frac{m-2}{2}$. Suppose that $\sigma $ is an
extended Born-Infeld field with the minus sign on $\mathbb{R}^m\, ,$ satisfying
\begin{equation}
|d\sigma |^2\leq 1-\frac{(q+1)^2}{(m-q-1)^2}
\tag{7.15}
\end{equation}
Then
\begin{equation}
\frac 1{\rho _1^{\frac {m}{q+1} }}\int_{B_{\rho
_1}(x_0)} 1-\sqrt{1-|d\sigma |^2}  \quad dx \leq \frac 1{\rho
_2^{\frac {m}{q+1} }}\int_{B_{\rho _2}(x_0)} 1-\sqrt{ 1-|d\sigma
|^2} \quad dx \tag{7.16}
\end{equation}
for any $0<\rho _1\leq \rho _2$.
\end{proposition}

\begin{proof}  Let $F(t)=1-\sqrt{1-2t}$. For the distance function
$r$ on $\Bbb{R}^m$, we have
\begin{equation}
Hess(r)=\frac 1r[g-dr\otimes dr]  \tag{7.17}
\end{equation}
where $g$ is the standard Euclidean metric. Taking $X=r\nabla r$,
using (4.6) and (7.17), we have at those points $x \in \mathbb{R}^m\, ,$ where $d\sigma (x) \ne 0\, ,$
\begin{equation}
\aligned \langle S_{F,d\sigma },\nabla \theta
_X\rangle&=mF(\frac{|d\sigma
|^2}2)-qF^{\prime }(\frac{|d\sigma |^2}2)|d\sigma |^2 \\
&=\big(m-q\frac{F^{^{\prime }}(\frac{|d\sigma
|^2}2)|d\sigma |^2}{F( \frac{|d\sigma
|^2}2)}\big)F(\frac{|d\sigma |^2}2) \endaligned \tag{7.18}
\end{equation}
From (7.13), it is easy to see that (7.15) is equivalent to, for every $x \in \mathbb{R}^m\, ,$
\begin{equation}
m-q\frac{F^{\prime} (\frac{|d\sigma |^2}{2})|d\sigma
|^2}{F(\frac{ |d\sigma |^2}2)}= m-q(1 + \frac {1}{\sqrt{1-|d\sigma|^2}}) \ge \frac {m}{q+1}  \tag{7.19}
\end{equation}
which implies
\begin{equation}
\langle S_{F,d\sigma },\nabla \theta _X\rangle\geq \frac {m}{q+1} F(\frac{|d\sigma |^2}2) \quad \operatorname{on}\quad B_{\rho}(x_0)\tag{7.20}
\end{equation}
Therefore we can prove this Proposition by using (7.20) in the
same way as we prove Theorem 4.1, via $(4.10)\, .$ \end{proof}

\begin{corollary} In addition to the same hypotheses of
Proposition 7.3, if $\sigma $ satisfies
$$
\int_{B_\rho(x_0)} 1-\sqrt{1-|d\sigma |^2} \quad dx = o(\rho^{\frac {m}{q+1} })\quad \text{as
} \rho\rightarrow \infty
$$
then $d\sigma \equiv 0\, ,$ and $\sigma$ is exact. In particular, if $\sigma $ has finite $E
_{BI}^{-}-$energy, then $\sigma$ is exact.
\end{corollary}

\section{Generalized Yang-Mills-Born-Infeld
Fields (with the plus sign and with the minus sign) on Manifolds}
In [SiSiYa], L. Sibner, R. Sibner and Y.S. Yang consider a variational problem which is a generalization of the (scalar valued) Born-Infeld model and at the same time a quasilinear generalization of the Yang-Mills theory. This motivates the study of Yang-Mills-Born-Infeld
fields on $\Bbb{R}^4$, and they prove that a generalized self-dual
equation whose solutions are Yang-Mills-Born-Infeld fields has no
finite-energy solution except the trivial solution on $\Bbb{R}^4$. In this section, we introduce the following
\begin{definition}
The \emph{generalized Yang-Mills-Born-Infeld energy functional with the plus sign on a manifold} $M$ is the mapping $\mathcal{YM}_{BI}^{+} : \mathcal{C}\to \mathbb{R}^+\, $ given by
\begin{equation}
\mathcal{YM}_{BI}^{+}(\nabla )=\int_{M} \sqrt{1+||R ^\nabla
||^2}-1 \quad  dv\tag{8.1}
\end{equation}
and the \emph{generalized Yang-Mills-Born-Infeld energy functional with the minus sign on a manifold} $M$  is the mapping $\mathcal{YM}_{BI}^{-} : \mathcal{C}\to \mathbb{R}^+\, $ given by
\begin{equation}
\mathcal{YM}_{BI}^{-}(\nabla )=\int_{M} 1-\sqrt{1-|| R ^\nabla
||^2} \quad dv\tag{8.2}
\end{equation}
The associate curvature form $R ^\nabla $ of a critical connection $\nabla $ of
$\mathcal{YM}_{BI}^{+}$ $($resp. $\mathcal{YM}_{BI}^{-}$$)$ is called a generalized
\emph{Yang-Mills-Born-Infeld field with the plus sign} $($resp. \emph{with the
minus sign}$)$ on a manifold.
\end{definition}
By applying $F(t)=\sqrt{1+2t} -1\, $ and
$F(t)=1-\sqrt{1-2t}\, $ to Theorem 3.1, we obtain
\begin{corollary}
Every generalized Yang-Mills-Born-Infeld field $($with the plus sign or with the minus sign$)$ on a manifold satisfies an $F$-conservation law.
\end{corollary}

\begin{theorem} Let the radial curvature $K_r$ of $M$ satisfy one of the three conditions $(i), (ii)\, ,$ and $(iii)\, $ in Theorem 4.1 in which $p=2\, $ and $d_F=1\, .$ Let $R ^\nabla $ be a generalized
Yang-Mills-Born-Infeld field with the plus sign on $M$. If
$R ^\nabla $ satisfies the following growth
condition
$$
\int_{B_\rho(x_0)} \sqrt{1+|| R ^\nabla ||^2}-1 \quad dv = o(\rho^{\lambda})\quad \text{as
} \rho\rightarrow \infty
$$
where

\[
\lambda =\left\{
\begin{array}{cc}
m-4\frac \alpha \beta & \text{if }\quad K_r \text{ satisfies $($i$)$;} \\
m-4 & \text{if }\quad K_r \text{ satisfies $($ii$)$;} \\
m - (m-1)\frac B{2\epsilon} - 4 e^{\frac {A}{2\epsilon}} & \text{if }\quad K_r \text{
satisfies $($iii$)$}
\end{array}
\right.
\]
then its curvature $R ^\nabla \equiv 0$. In particular, if
$R ^\nabla $ has finite $\mathcal{YM}_{BI}^{+}$-energy, then $R ^\nabla \equiv
0$.
\end{theorem}
\begin{proof} By applying
Corollary 8.1 and $F(t)=\sqrt{1+2t}-1$ to Theorem 4.1 in which $d_F=1\, ,$ by Lemma 7.1(i), and $p=2\, ,$ for $R ^\nabla \in A^2(Ad P)$, the result follows immediately.
\end{proof}

\begin{theorem} Suppose $M$ has constant sectional
curvature $-\alpha ^2$ $($$\alpha ^2\geq 0$$)$. Let $R ^\nabla $ be a generalized
Yang-Mills-Born-Infeld field with the plus sign on $M$. If
$m>4$ and $R ^\nabla $ satisfies the following growth
condition
$$
\int_{B_\rho(x_0)} \sqrt{1+||R ^\nabla ||^2}-1 \quad dv = o(\rho^{m-4})\qquad as \qquad \rho\to \infty
$$
then its curvature $R ^\nabla \equiv 0$. In particular, if
$R ^\nabla $ has finite $\mathcal{YM}_{BI}^{+}$-energy, then $R ^\nabla \equiv
0$.
\end{theorem}
\begin{proof} This follows at once by applying $\alpha = \beta$ in conditions (i) and (ii) of Theorem 8.1.
\end{proof}
\begin{corollary}
Let $R ^\nabla $ be a
Yang-Mills-Born-Infeld field with the plus sign on $\Bbb{R}^m$. If
$m>4$ and $R ^\nabla $ satisfies the following growth
condition
$$
\int_{B_\rho(x_0)} \sqrt{1+||R ^\nabla ||^2}-1 \quad dx = o(\rho^{m-4})\quad \text{as
} \rho\rightarrow \infty
$$
then its curvature $R ^\nabla \equiv 0$. In particular, if
$R ^\nabla $ has finite $\mathcal{YM}_{BI}^{+}$-energy, then $R ^\nabla \equiv
0$.
\end{corollary}
If we replace $d\sigma $ with $R ^\nabla $ and set $q=2$ in Proposition 7.3, by a similar argument, we obtain the
following

\begin{proposition} Let $R ^\nabla $ be a
Yang-Mills-Born-Infeld field with the minus sign on $\Bbb{R}^m$.
Suppose $m>6$ ,
$$
|| R^\nabla ||^2\leq \frac {m^2-6m}{m^2-6m+9}
$$
and
$$
\int_{B_\rho(x_0)} 1-\sqrt{1-|| R ^\nabla ||^2} \quad dx = o(\rho^{\frac m2})\quad \text{as
} \rho\rightarrow \infty
$$
Then $R ^\nabla =0$.
\end{proposition}

It is well-known that there are no nontrivial Yang-Mills fields in
$\Bbb{R}^m$ with finite Yang-Mills-energy for $m\geq 5$ (in contrast with
$\Bbb{R}^4$, where the problem is conformally invariant and one
obtains Yang-Mills fields with finite $\mathcal{YM}$-energy by pullback from $S^4$ (cf. [JT])).
In Corollary 8.2, for the case $m\geq 5$, we obtain a similar result
for Yang-Mills-Born-Infeld field (with the plus sign) on $\Bbb{R}^m$. It's
natural to ask if there exists a nontrivial Yang-Mills-Born-Infeld field (with the plus sign) on $\Bbb{R}^4$ with
finite $\mathcal{YM}_{BI}^{+}$-energy.

\section{Generalized Chern type Results on Manifolds}

A Theorem of Chern states that every entire
graph $ x_{m+1}$ $=f(x_1,...,x_m)$ on $\Bbb{R}^m$ of constant mean
curvature is minimal in $\Bbb{R}^{m+1}$.
In this section, we view functions as $0$-forms and consider the following constant mean curvature
type equation for $p-$forms $\omega$ on $\Bbb{R}^m$ $(p < m)$ and on manifolds with the global doubling property by a different approach
(being motivated by the work in [We1,2] and [LWW]):
\begin{equation}
\delta \left( \frac{d\omega }{\sqrt{1+|d\omega |^2}}\right)
=\omega _0 \tag{9.1}
\end{equation}
where $\omega _0$ is a constant $p-$form (Thus when $p=0$, (9.1) is
just the equation describing graphic hypersurface with constant
mean curvature). Equivalently, (9.1) may be written as
\begin{equation}
d*\left( \frac{d\omega }{\sqrt{1+|d\omega |^2}}\right) =\xi_0
\tag{9.2}
\end{equation}
where $\xi_0$ is a constant $(m-p)-$form.

\begin{theorem} Suppose $\omega $ is a solution of $($9.2$)$ on
$\Bbb{R}^m$ . Then $\xi_0=0$.
\end{theorem}

\begin{proof} Obviously, for every $(m-p)-$plane $\Sigma $ in
$\Bbb{R}^m$, there exists a volume element $d\Sigma $ of $\Sigma
$, such that $\xi_0|_\Sigma =c\, d\Sigma $, for some constant $c$. Let $i:\Sigma\hookrightarrow \Bbb{R}^m$ be the
inclusion mapping. If follows from (9.2) and
Stokes' Theorem that for every ball $B(x_0,r)$ of radius $r$
centered at $x_0$ in $\Sigma \subset \Bbb{R}^m$, and its boundary
$\partial B(x_0,r)$ with the surface element $dS$, we have
$$
\aligned 0 &\leq |c|\omega _{m-p}r^{m-p}\\&=\bigg |\int_{B(x_0,r)}c\, d\Sigma
\bigg |\\& = \bigg |\int_{B(x_0,r)}i^{*}\xi_0\bigg |\\
&= \bigg |\int_{B(x_0,r)}
di^{*}\big(*\frac{d\omega
}{\sqrt{1+|d\omega |^2}}\big)\bigg |\\
& = \bigg |\int_{\partial B(x_0,r)}i^{*}* \frac{d\omega }{\sqrt{1+|d\omega |^2}}\bigg | \\
\ &\leq \int_{\partial B(x_0,r)}\bigg |\frac{d\omega }{\sqrt{1+|d\omega
|^2}}\bigg |\, dS \\& \leq (m-p)\omega _{m-p}r^{m-p-1}
\endaligned
$$
where $\omega _{m-p}$ is the volume of the unit ball in $\Sigma $.
Hence we get
\begin{equation}
0\leq |c|\leq \frac{m-p}r  \tag{9.3}
\end{equation}
which implies that $c=0$ by letting $r\rightarrow \infty $.
\end{proof}

This generalizes the work of Chern:

\begin{corollary}$([$Che$])$ Let $p=0$ in Theorem 9.1. Then the
graph of $\omega $ over $R^m$ is a minimal hypersurface in
$\Bbb{R}^{m+1}$.
\end{corollary}
\begin{proof} As $p=0$, we may assume that $\omega =f$ for some
function $f$ on $\Bbb{R}^m $. Then (9.1) is equivalent to
\begin{equation}
\operatorname{div} \left( \frac{\nabla f}{\sqrt{1+|\nabla f|^2}}\right) =c
\tag{9.4}
\end{equation}
where $c\ $is a constant. Now the assertion follows from Theorem
9.1.
\end{proof}

\begin{corollary} Let $p=0$ and $m\leq 7$ in Theorem 9.1.
Then the graph of $\omega $ over $\Bbb{R}^m $ is a hyperplane in
$\Bbb{R}^{m+1}$.
\end{corollary}

\begin{proof} This follows at once from Corollary 9.1 and Bernstein
Theorems for minimal graphs (cf. [Be], [Al], [Gi] and [Si]).
\end{proof}

\begin{corollary} Let $p=0$ and $|\nabla \omega| (x)\le \beta\, $ $($for all $x\in \Bbb{R}^m\, ,$ where $\beta >0$ is a constant$)$ in Theorem 9.1.
Then the graph of $\omega $ over $\Bbb{R}^m $ is a hyperplane in
$\Bbb{R}^{m+1}\, ,$ for all $m \ge 2$.
\end{corollary}

\begin{proof} This follows at once from Corollary 9.1 and Harnack's Theorem due to Moser
(cf. [Mo], p.591).
\end{proof}

In fact, we can give a further generalization.

\begin{theorem} Let $\omega $ be a differential form of
degree $p$ on $\Bbb{R}^m$, satisfying
\begin{equation}
d*\left( \frac{d\omega }{\sqrt{1+|d\omega |^2}}\right) =\xi
\tag{9.5}
\end{equation}
where $\xi$ is a differential form of degree $m-p$ on $\Bbb{R}^m$.
Suppose there exists an $(m-p)-$plane $\Sigma $ in $\Bbb{R}^m$,
with the volume element $d\Sigma $, such that $\xi|_\Sigma
=g(x)d\Sigma $, off a bounded set $K$ in $ \Sigma $, where $g$ is
a continuous function on $\Sigma\backslash K$ with $c = \inf_{x\in
\Sigma \backslash K} |g(x)|$ . Then $c=0$.
\end{theorem}

\begin{proof} We consider two cases:

Case 1. $g$ assumes both positive and negative values: By the
intermediate value theorem, $g$ assumes value $0$ at some point,
and thus $c=\inf_{x\in\Sigma\backslash K}|g(x)|=0$.

Case 2. $g$ is a nonpositive or nonnegative function: Since
$K\subset \Sigma $ is bounded, choose a sufficiently large $r_0<r$
so that $K\subset B(x_0,r_0)$, where $B(x_0,r_0)$ is the ball of
radius $r_0$ centered at $x_0$ in $\Sigma \subset \Bbb{R}^m$. Let
$0\leq \psi \leq 1$ be the cut off function such
that $\psi \equiv 1$ on $B(x_0,r_0)$ and $\psi \equiv 0$ off
$B(x_0,2r)\subset \Sigma$, and $|\nabla \psi |\leq \frac{C_1}r$(cf. also
Lemma 1 in [We1]). Let $i:\Sigma\hookrightarrow \Bbb{R}^m$ be the
inclusion mapping. Multiplying (9.5) by $\psi $, and applying the
divergence theorem, we have

$$
\aligned
c\omega _{m-p}(r^{m-p}-r_0^{m-p}) &\leq
\bigg |\int_{B(x_0,r)\backslash B(x_0,r_0)}\psi (x)g(x)d\Sigma\bigg |\\
&= \bigg |\int_{B(x_0,r)\backslash B(x_0,r_0)}\psi i^{*}\xi\bigg |\\
&= \bigg |\int_{B(x_0,r)\backslash B(x_0,r_0)}\psi
di^{*}\big(*(\frac{d\omega
}{\sqrt{1+|d\omega |^2}})\big)\bigg |\\
&\leq \int_{B(x_0,2r)\backslash B(x_0,r_0)}|\nabla \psi
|\bigg |\frac{d\omega }{\sqrt{1+|d\omega |^2}} \bigg |\,  d\Sigma \\
&\qquad + \int_{\partial B(x_0,r_0)}\bigg |\frac{d\omega }{\sqrt{1+|d\omega |^2}}\bigg |\, dS\\
&\leq \omega
_{m-p}C_12^{m-p}r^{m-p-1} + (m-p)\omega _{m-p}r_0^{m-p-1}\endaligned
$$
where $\omega _{m-p}$ is the volume of the unit ball in $\Sigma $.
Hence
$$
0\leq c\omega _{m-p}(1-\frac{r_0^{m-p}}{r^{m-p}})\leq \frac{\omega
_{m-p}C_12^{m-p}}r +
\frac{(m-p)\omega _{m-p}r_0^{m-p-1}}{r^{m-p}}
$$
implies that $c=0$ by letting $r\rightarrow \infty $.
\end{proof}

\begin{corollary} There does not exist a solution of $($9.5$)$
such that $\xi|_\Sigma =g(x)d\Sigma $ , off a bounded set $K$ in
some $(m-p)-$plane $\Sigma $ in $\Bbb{R}^m$ with $c > 0$, where
$g$ is a continuous real-valued $($not necessary nonnegative or
nonpositive$)$ function, and $c = \inf_{x\in \Sigma \backslash K}
|g(x)|$.

\end{corollary}

\begin{corollary} Let $f$ be a function satisfying
$$
\operatorname{div}\left(\frac{\nabla f}{\sqrt{1+|\nabla f|^2}}\right)=c
$$
off a bounded subset $K\subset \Bbb{R}^m$, where $c\equiv
$const. Then $c=0$. In particular, every graph of $f$ of constant
mean curvature off a cylinder $\Bbb{R}^m\backslash
(B(x_0,r_0)\times \Bbb{R})$ is minimal.
\end{corollary}

\begin{proof} This follows at once from Theorem 9.2 in which
$\omega =df$ and $p=0$ . In particular, we choose $K=B(x_0,r_0)$.
\end{proof}

\begin{remark} This result, Corollary 9.4, recaptures
Corollary 9.1, a theorem of Chern, in which $K$ is an empty set.
Notice that Chern's result was also generalized to graphs with
higher codimension and parallel mean curvature in Euclidean space
by Salavessa [Sa1].
\end{remark}

Next we consider the following equation
\begin{equation}
\delta \left( \frac{d\sigma }{\sqrt{1-|d\sigma |^2}}\right) =\rho
_0 \tag{9.6}
\end{equation}
which generalizes the constant mean curvature equation for
spacelike hypersurfaces.

\begin{theorem} Let $\sigma $ be a differential form of
degree $q$ $($$\leq m-1$$)$ on $\Bbb{R}^m$, satisfying
\begin{equation}
d*\left( \frac{d\sigma }{\sqrt{1-|d\sigma |^2}}\right) =\tau _0
\tag{9.7}
\end{equation}
where $\tau _0$ is a constant $(m-q)-$form. If
\begin{equation}
\frac 1{\sqrt{1-|d\sigma |^2}}=o(r)  \tag{9.8}
\end{equation}
where $r$ is the distance from the origin, then $\tau _0=0$.
\end{theorem}

\begin{proof} Obviously, for every $(m-q)-$plane $\Sigma $ in
$\Bbb{R}^m$, there exists a volume element $d\Sigma $ of $\Sigma
$, such that $\tau _0|_\Sigma =c\, d\Sigma $. Let
$i:\Sigma\hookrightarrow \Bbb{R}^m$ be the inclusion mapping. For
every ball $ B(x_0,r)$ of radius $r$ centered at $x_0$ in $\Sigma
\subset \Bbb{R}^m$, and its boundary $\partial B(x_0,r)$, by using
(9.7) and Stokes' Theorem, we have
$$
\aligned |c|\omega _{m-q}r^{m-q} &=\bigg |\int_{B(x_0,r)}c\, d\Sigma
\bigg |\\
&=\bigg |\int_{\partial
B(x_0,r)}i^{*}\big(*(\frac{d\sigma }{\sqrt{1-|d\sigma |^2}})\big)\bigg | \\
&\leq \int_{\partial B(x_0,r)}\bigg |\frac{d\sigma }{\sqrt{1-|d\sigma
|^2}} \bigg |\,  dS \\
&\leq (m-q)\sup_{\partial B(x_0,r)}\{\frac 1{\sqrt{1-|d\sigma |^2}
}\}\omega _{m-q}r^{m-q-1}\endaligned
$$
where $\omega _{m-q}$ is the volume of the unit ball in $\Sigma $.
Hence
$$
|c|\leq \frac{m-q}r\sup_{\partial B(x_0,r)}\{\frac
1{\sqrt{1-|d\sigma |^2}}\}
$$
implies that $c=0$ by letting $r\rightarrow \infty $.
\end{proof}

\begin{remark} (1) When $q=0$, (9.6) describes spacelike
graphic hypersurface with constant mean curvature. It is known
that $\frac 1{\sqrt{1-|d\sigma |^2}}$ is bounded iff the Gauss
image of the hypersurface is bounded (cf. [Xi2,3]). Such kind of
Chern type results under growth conditions were obtained in [Do],
[Sa2] for spacelike graphs as well. (2) A similar generalized Chern type result
can be established for the following more general equation
$$
\delta ^\nabla \big(F^{\prime }(\frac{|d^\nabla \sigma |^2}2)d^\nabla
\sigma \big) = \rho _0
$$
\end{remark}
Using a different technique or idea (cf. [We1,2], [LWW]), one can
extend the above results to complete noncompact manifold $M$ that
has the global doubling property, i.e., $\exists D(M)
> 0\, $ such that $\forall r
>0\, ,$ $\forall x \in M\,$
\begin{equation}
 Vol(B(x,2r)) \le D(M) Vol(B(x,r))\tag{9.9}
\end{equation}
Examples of complete manifolds with the global doubling property
include complete noncompact manifolds of nonnegative Ricci
curvature, in particular Euclidean space $\Bbb{R}^m$.
\begin{theorem} Let $\omega $ be a differential form of
degree $p$ on $M$ that  has the global doubling property, and
satisfies
\begin{equation}
d*\left( \frac{d\omega }{\sqrt{1+|d\omega |^2}}\right) =\xi
\tag{9.10}
\end{equation}
where $\xi$ is a differential form of degree $m-p$ on $M$. Suppose
there exists an $(m-p)-$ dimensional submanifold $\Sigma $ in $M$,
with the volume element $d\Sigma $, such that $\xi|_\Sigma
=g(x)d\Sigma $, off a bounded set $K$ in $ \Sigma $, where $g$ is
a continuous function on $\Sigma\backslash K$ with $c = \inf_{x\in
\Sigma \backslash K} |g(x)|$ . Then $c=0$.
\end{theorem}
\begin{proof} Proceed as in the proof of Theorem 9.2, it suffices
to show the result holds for $g \ge 0$ or $g \le 0\, .$ Let
$K\subset B(x_0,r_0)$, where $B(x_0,r_0)$ is the geodesic ball of
radius $r_0$ in $M\, ,$ centered at $x_0\, .$ Let $0\leq \psi \leq
1$ be the cut off function such that $\psi \equiv 1$ on
$B(x_0,r_0)$ and $\psi \equiv 0$ off $B(x_0,2r)$, and $|\nabla
\psi |\leq \frac{C_1}r$(cf. also Lemma 1 in [We1]). Let
$i:\Sigma\hookrightarrow M$ be the inclusion mapping. Then
multiplying both sides of the equation (9.10) by $\psi\, ,$
integrating over the annulus $B(x_0,2r)\backslash B(x_0, r_0)
(\subset M\backslash K )\, ,$ and applying Stokes' Theorem, we
have
\begin{equation}
\aligned
 c(Vol(B(x_0, r)) - Vol(B(x_0, r_0)) &\leq
\bigg |\int_{B(x_0,r)\backslash B(x_0,r_0)}\psi (x)g(x)\, d\Sigma\bigg |\\
& \le \bigg| \int _{B(x_0,2r)\backslash B(x_0, r_0)} \psi g\, (x)
\, d\Sigma\bigg| \\
&= \bigg |\int_{B(x_0,2r)\backslash B(x_0,r_0)}\psi i^{*}\xi\bigg |\\
&= \bigg|\int_{B(x_0,2r)\backslash B(x_0,r_0)}\psi
di^{*}\big(*(\frac{d\omega
}{\sqrt{1+|d\omega |^2}})\big)\bigg|\\
&\leq \int_{B(x_0,2r)\backslash B(x_0,r_0)}|\nabla \psi
|\bigg |\frac{d\omega }{\sqrt{1+|d\omega |^2}}\bigg |d\Sigma\\
&\qquad + \int_{\partial B(x_0,r_0)}\bigg |\frac{d\omega }{\sqrt{1+|d\omega |^2}}\bigg |\,  dS\\
&\leq Vol(\partial B(x_0, r_0)) +  \frac {C_1 Vol(B(x_0, 2r))}r
\endaligned \tag{9.11}
\end{equation}
where $dS$ is the area element of $\partial B(x_0,r)\, .$ Hence,
dividing (9.11) by $Vol(B(x_0, r))\, ,$  one has
$$
c (1-\frac {VolB(x_0, r_0)}{Vol(B(x_0, r))})  \le \frac {
Vol(\partial B(x_0, r_0))}{Vol(B(x_0, r))} + \frac { C_1
D(M)}{r}\rightarrow 0
$$
as $r\to\infty\, ,$ since $M$ has infinite volume (by Lemma 5.1 in [LWW]).

\end{proof}

\section{Appendix: A Theorem on $\mathcal{E}_{F,g}$-energy Growth}
In this Appendix, we provide a theorem on $\mathcal{E}_{F,g}$-energy growth, with examples (cf. Examples 10.1 and 10.2). These in particular, imply that our growth assumptions (5.1) and (5.4) in Liouville type results are weaker than the existing growth conditions such as {\it finite} $\mathcal{E}_{F,g}$-energy, {\it slowly divergent} $\mathcal{E}_{F,g}$-energy (cf. (5.3)), (10.6), and (10.7).

We say that $f(r) \sim g(r)$ as  $r \to \infty \, ,$ if $\, \limsup _{r \to \infty} \frac {f(r)}{g(r)} = 1\, ,$ and $f(r) \nsim g(r)\, $ as  $\, r \to \infty \, ,$ otherwise.
We say that $f(r)\asymp g(r)$ for large $r\, ,$ if there exist positive constants $k_1$ and $k_2$ such that $k_1 g(r) \le f(r) \le k_2 g(r)\, $ for all large $r\, ,$ and $f(r)\not\asymp g(r)$ for large $r$ otherwise.

\begin{lemma}
Let $\psi(r) > 0$ be a continuous function such that
\begin{equation}
\int_{\rho_0}^\infty \frac{dr}{r\psi (r)}=+\infty\tag{5.2}
\end{equation}
for some $\rho_0>0\, .$
Then

$($i$)$ $\psi(r)$ can not go to infinity faster than $r^\lambda\, ,$ i.e., $\lim_{r \to \infty} \frac {\psi(r)}{r^\lambda} \ne \infty \, ,$ for any $\lambda > 0\, .$

$($ii$)$ If $\lim_{r \to \infty} \frac {\psi(r)}{r^\lambda}\, $ exists for some $\lambda > 0\, ,$ then
\begin{equation}
\lim_{r \to \infty} \frac {\psi(r)}{r^\lambda} = 0\, ,\tag{10.1}
\end{equation}
$f(r) \nsim g(r)\, ,$ and $\psi(r) \, \not\asymp r^{\lambda}\, .$
\end{lemma}

\begin{proof} Suppose on the contrary, i.e. $\lim_{r \to \infty} \frac {\psi(r)}{r^\lambda} = c < \infty,\, \text{where}\quad  c \ne 0\, $ (resp. $\quad \lim_{r \to \infty} \frac {\psi(r)}{r^\lambda} = \infty\, $ ). Then there would exist $\rho_1 > 0$ such that if $r \ge \rho_1,\quad \psi(r) > \frac c2 r^{\lambda}\, (\text {resp.}\quad   \psi(r) > k r^{\lambda}\, , \operatorname{where}\, k > 0\, $ is a constant.$)\, $ This would lead to $$ \int_{\rho_1}^\infty \frac{dr}{r\psi (r)}\le \frac 2c \int_{\rho_1}^\infty \frac{dr}{r^{1+\lambda}}\, \bigg(resp.\quad   k \int_{\rho_1}^\infty \frac{dr}{r^{1+\lambda}}\, \bigg) < \infty\, ,$$
contradicting (5.2), by the continuity of $\psi (r)$ if $\rho_0 < \rho_1\, .$
\end{proof}

\begin{theorem}
Let $\omega \in A^p(\xi)$ have slowly divergent $\mathcal{E}_{F,g}-$energy. That is,
\begin{equation}
\lim_{\rho\rightarrow \infty }\int_{B_\rho(x_0)}\frac{F(\frac{|\omega
|^2}2)}{\psi (r(x))} \, dv <\infty\tag{5.3}
\end{equation}
for some continuous function $\psi(r) > 0$ satisfying (5.2). Then

$($i$)$  For any $\lambda > 0\, ,$ $\lim_{r \to \infty} \frac {\psi(r)}{r^\lambda} \ne \infty \, .$

$($ii$)$ If $\lim_{r \to \infty} \frac {\psi(r)}{r^\lambda}\, $ exists for some $\lambda > 0\, ,$ then \begin{equation}
\int_{B_\rho(x_0)}F(\frac{|\omega |^2}2)\, dv = o(\rho^\lambda )\quad \text{as
} \rho\rightarrow \infty.\tag{5.1}\end{equation}\end{theorem}

\begin{proof}
In view of Lemma 10.1 and (10.1), we have for every $\epsilon > 0\, ,$
there exists $\rho_2 > 0\, ,$ such that if $r > \rho_2\, ,$ then \begin{equation}\psi(r) < \frac {\epsilon}{2(L+1)}r^{\lambda}\, ,\tag{10.2}
\end{equation} where $L := \lim_{\rho\rightarrow \infty }\int_{B_\rho(x_0)}\frac{F(\frac{|\omega
|^2}2)}{\psi (r(x))} dv\, $ (by assumption $0 \le L < \infty\, $). Hence by the definition of $L\, ,$ there exists $\rho_3>0$ such that if $\rho > \rho_3\, ,$ then
\begin{equation}
\int_{B_\rho(x_0)}\frac{F(\frac{|\omega
|^2}2)}{\psi (r(x))} dv < L + 1\tag{10.3}
\end{equation}
Since $\lim_{\rho \to \infty} \frac 1{\rho^{\lambda}} \int_{B_{\rho_2}(x_0)}F(\frac{|\omega |^2}2)\, dv =0$ , we have for every
$\epsilon > 0\, ,$
there exists $\rho_4 > 0\, ,$ such that if $\rho > \rho_4\, ,$ then \begin{equation}\frac 1{\rho^{\lambda}}\int_{B_{\rho_2}(x_0)}F(\frac{|\omega |^2}2)\, dv < \frac {\epsilon}{2}\, .\tag{10.4}
\end{equation}
It follows that for every
$\epsilon > 0\, ,$ one can choose $\rho_5 = \max \{\rho_2, \rho_3, \rho_4\}\, ,$ such that if $\rho > \rho_5\, ,$ then via (10.2) (10.3) and (10.4), we have
\begin{equation}
\aligned \frac 1{\rho^{\lambda}}\int_{B_{\rho}(x_0)}F(\frac{|\omega |^2}2)\, dv & = \frac 1{\rho^{\lambda}}\int_{B_{\rho_2}(x_0)}F(\frac{|\omega |^2}2)\, dv + \int_{B_{\rho}(x_0)\backslash B_{\rho_2}(x_0)}\frac{F(\frac{|\omega
|^2}2)}{\psi (r(x))}\frac {\psi (r(x))}{\rho^{\lambda}}\, dv\\
&<\frac {\epsilon}{2} + \frac {\epsilon}{2(L+1)}\int_{B_{\rho}(x_0)\backslash B_{\rho_2}(x_2)}\frac{F(\frac{|\omega
|^2}2)}{\psi (r(x))}\frac {r^{\lambda}}{\rho^{\lambda}}\, dv\\
&\le \frac {\epsilon}{2} + \frac {\epsilon}{2(L+1)}\int_{B_{\rho}(x_0)}\frac{F(\frac{|\omega
|^2}2)}{\psi (r(x))} \, dv\, \\
& < \frac {\epsilon}{2} + \frac {\epsilon}{2} = \epsilon
\endaligned \tag{10.5}
\end{equation}
That is, (5.1) holds.
\end{proof}

\begin{example}
Let $\omega \in A^p(\xi)$ have the growth rate \begin{equation}
\lim_{\rho\rightarrow \infty }\int_{B_\rho(x_0)}\frac{F(\frac{|\omega
|^2}2)}{(\ln r(x))^q} \, dv <\infty\tag{10.6}\end{equation}
  for some number $q \le 1\, .$ Then $\omega $ has slowly divergent $\mathcal{E}_{F,g}-$energy $(5.3)$, as $\psi(r)=(\ln r)^q$ satisfies (5.2) for any number $q \le 1\, .$
Furthermore, as an immediate consequence of Theorem 10.1, $\omega $ has the growth rate \begin{equation}
\int_{B_\rho(x_0)}F(\frac{|\omega |^2}2)\, dv = o(\rho^\lambda )\quad \text{as
} \rho\rightarrow \infty\tag{5.1}\end{equation}
 for any $\lambda > 0\, .$
\end{example}

The following is an example of $\psi(r)$ that does not satisfy (5.2), yet $\omega$
has the growth rate (5.1):

\begin{example}
Let $\omega \in A^p(\xi)$ have the growth rate \begin{equation}
\lim_{\rho\rightarrow \infty }\int_{B_\rho(x_0)}\frac{F(\frac{|\omega
|^2}2)}{(\ln r(x))^{q'}} \, dv <\infty\tag{10.7}\end{equation}
  for some number $q' >1\, .$ Then $\psi(r)=(\ln r)^{q^{\prime}}$ does not satisfy $(5.2)$ for any number $q^{\prime} > 1\, .$
Since $(\ln \rho)^{q^{\prime}}$ goes to infinity slower than $\rho^\lambda\, $ for any $q^{\prime}, \lambda > 0\, ,$ it is evident that $\omega $ has the growth rate $(5.1)\, ,$ via $(10.3)$ for any $\lambda > 0\, .$
\end{example}

\vskip 0.3 true cm

{\bf Acknowledgments}: The authors wish to thank Professors
J.G. Cao and Z.X. Zhou, and Ye Li for their helpful discussions.



\begin{thebibliography}{SiSiYa}

\bibitem[A]{A} W.K. Allard, {\em On the first variation of a varifold}, Ann. of Math. (2) 95 (1972), 417--491.

\bibitem[Al]{Al}  F.J. Almgren, JR., {\it Some interior regularity theorems for
minimal surfaces and extension of Bernstein's theorem}, Ann. of Math. (2) 84
(1966), 277-292

\bibitem[Ar]{Ar}  M. Ara, {\it Geometry of $F-$harmonic maps}, Kodai Math. J., 22
(1999), 243-263

\bibitem[Ba]{Ba}  P. Baird, {\it Stress-energy tensors and the Lichnerowicz
Laplacian}, J. Geom. Phys. 58 (2008), no. 10, 1329-1342.

\bibitem[Be]{Be}  S. Bernstein, {\it Sur un theoreme de geometrie at ses
application aux equations aux derivees partielles du type elliptique}, Comm.
Soc. Math. Kharkov (2) 15 (1915-C1917), 38-45

\bibitem[BE]{BE}  P. Baird and J. Eells, {\it A conservation law for harmonic maps},
in: Geometry Symposium, Utrecht 1980, in: Lecture notes in Mathematics, Vol.
894, Springer (1982), 1-25


\bibitem[BI]{BI}  M. Born and L. Infeld, {\it Foundation of a new field theory},
Proc. R. Soc. London Ser. A., 144 (1934), 425-451

\bibitem[BN]{BN}  B.M. Barbashov, V.V. Nesterenko, {\it Introduction to the
relativistic string theory}, Singapore World Scientific (1990)

\bibitem[Ca]{Ca} E. Calabi, {\it Examples of Bernstein problems for some nonlinear equations},  Global Analysis,  (Proc. Sympos. Pure Math., Vol. XV, Berkeley, Calif., 1968) (1970), 223-230

\bibitem[Ch]{Ch}  Q. Chen, {\it Stability and constant boundary-value problems
of harmonic maps with potential}, J. Aust. Math. Soc. (Series A)68 (2000)
145-154

\bibitem[Che]{Che} S.S. Chern {\it On the curvature of a piece of
hypersurface in Euclidean space}, Abh. Math. Sem. Hamburg,
29 (1965), 77-91.

\bibitem[CZ]{CZ} Q. Chen and Z.R. Zhou, {\it On gap properties and instabilities of $p$-Yang-Mills fields},
Canad. J. Math. 59 (2007), no. 6, 1245--1259.

\bibitem[DG]{DG}  S. Deser and G.W. Gibbons, {\it Born-Infeld actions}, Class.
Quantun 15 (1998), L35-9

\bibitem[Di]{Di}  Q. Ding, {\it The Dirichlet problem at infinity for manifolds
of nonpositive curvature}, Gu, C. H. (ed.) et al., Differential geometry,
Proc. of the Symp. in honour of Prof. Su Buchin on his 90th birthday
(Shanghai), World Sci. (1993), 48-58.

\bibitem[Do]{Do}  Y.X. Dong, {\it Bernstein theorems for spacelike graphs with
parallel mean curvature and controlled growth}, J. Geom. Phys. 58
(2008), no. 3, 324-333

\bibitem[EF1]{EF1}  J.F. Escobar, A. Freire, {\it The spectrum of the Laplacian
of manifolds of positive curvature}, Duke Math. J. 65 (1992), 1-21

\bibitem[EF2]{EF2}  J.F. Escobar, A. Freire, {\it The differential form spectrum
of manifolds of positive curvature}, Duke Math. J. 69 (1993), 1-41

\bibitem[EFM]{EFM}  J.F. Escobar, A. Freire and M. Min-Oo, {\it $L^2$ vanishing
theorems in positive curvature}, Indiana Univ. Math. J., Vol. 42, No. 4
(1993), 1545-1554

\bibitem[FF]{FF} H. Federer and  W. H. Fleming,  {\it   Normal and
integral currents}, Ann. of Math. 72, (1960), 458-520

\bibitem[F1]{Fl} W.H. Fleming, {\it On the oriented Plateau problem}, Rend. Circ. Mat. Palermo (2) 11, (1962), 69--90

\bibitem[Gi]{Gi}  E. de Giorgi, {\it Una estensione del theorema di Bernstein},
Ann. Scuola Norm. Sup. Pisa (3) 19 (1965), 79-85

\bibitem[GRSB]{GRSB}  W.D. Garber, S.N.M. Ruijsenaars, E. Seiler and D. Burns,
{\it On finite action solutions of the nonlinear $\sigma-$model}. Ann. Phys. 119, 305-325
(1979)

\bibitem[GW]{GW}  R.E. Greene and H. Wu, {\it Function theory on manifolds which
posses a pole}, Lecture Notes in Math. 699 (1979), Springer-Verlag

\bibitem[HL]{HL} R. Hardt and F.H. Lin {\it  Mapping minimizing
the $L^p$ norm of the gradient}, XL  Comm on. Pure and Applied
Math. (1987), 555-588

\bibitem[Hi]{Hi}  S. Hildebrandt, {\it Liouville theorems for harmonic mappings,
and an approach to Beinstein theorems}. Ann. Math. Stud. 102, 107-131 (1982)

\bibitem[Hu1]{Hu1}  H.S. Hu, {\it On the static solutions of massive Yang-Mills
equations}, Chinese Annals of Math. 3(1982), 519-526

\bibitem[Hu2]{Hu2}  H.S. Hu, {\it A nonexistence theorem for harmonic maps with
slowly divergent energy}, Chinese Ann. of Math. Ser.B, 5(4), (1984), 737-740

\bibitem[JT]{JT}  A. Jaffe, C. Taubes, {\it Vortices and Monopoles: Structures of
Static Gauge Theories}, Birkhauser, Boston, 1980

\bibitem[Ka]{Ka}  M. Kassi, {\it A Liouville Theorem for $F-$harmonic maps with
finite $F-$energy}, Electronic Journal of Differential Equations, 15 (2006),
1-9

\bibitem[Ke]{Ke}  S.V. Ketov, {\it Many faces of Born-Infeld theory},
hep-th/0108189

\bibitem[KW]{KW}  H. Karcher and J.C. Wood, {\it Non-existence results and growth
properties for harmonic maps and forms}, J. Reine Angew. Math. 353 (1984)
165-180

\bibitem[La]{La} H. B. Lawson,Jr. {\it The theory of gauge fields in four dimensions}, CBMS Regional Conference Series in Mathematics, 58. (1985)

\bibitem[LY]{LY}  F. H. Lin, Y.S. Yang, {\it Gauged harmonic maps,
Born-Infeld electromagnetism, and Magnetic Vortices},
Communications on Pure and Applied Mathematics, Vol. LVI.
1631-1665 (2003)

\bibitem[Li1]{Li1}  J.C. Liu, {\it Vanishing theorem for $L^p-$forms valued on
vector bundle} (Chinese), J. of Northwest Normal University (Natural
Science), Vol.38 No.4 (2004), 20-22

\bibitem[Li2]{Li2}  J.C. Liu, {\it Constant boundary-valued problems for $p-$%
harmonic maps with potential}, J. of Geom. and Phys. 57 (2007) 2411-2418

\bibitem[LL1]{LL1}  J.C. Liu, C.S. Liao, {\it The energy growth property for $p-$%
harmonic maps}, J. of East China Normal Univ. (Natural Sci.) (2004) no.2,
19-23

\bibitem[LL2]{LL2}  J,C. Liu, C.S. Liao, {\it Nonexistence Theorems for $F-$%
harmonic maps} (Chinese) J. of Math. Reserach and Exposition, Vol. 26 (2006),
311-316

\bibitem[LSC]{LSC}  M. Lu, X.W. Shen and K.R. Cai, {\it Liouville type Theorem
for $p-$forms valued on vector bundle} (Chinese), J. of Hangzhou Normal Univ.
(Natural Sci.) (2008) no.7, 96-100

\bibitem[Lu]{Lu} S. Luckhaus {\em Partial H\"older continuity for
minima of certain energies among maps into a Riemannian manifold},
Indiana Univ. Math. J. 37 (1988)349-367

\bibitem[LW]{LW}  P. Li and J.P. Wang, {\it Finiteness of disjoint minimal graphs},
Math. Research Letters. 8, (2001), 771-777

\bibitem[LWe]{LWe} J.F. Li and S.W. Wei, {\em A $p$-harmonic approach to
generalized Bernstein problem}, Commun. Math. Anal. Conference 1, (2008) 35-39

\bibitem[LWW]{LWW} Y.I. Lee, A.N. Wang and S.W. Wei, {\em On a generalized $1$-harmonic equations and the inverse mean
curvature flow}, to appear in J. Geom. Phys.

\bibitem[Mo]{Mo} J. Moser, {\em On Harnack's theorem for elliptic differential equations}, XIV  Comm on Pure and Applied
Math. (1961), 577-591

\bibitem[RS]{RS}  M. Rigoli, A. Setti, {\em Energy estimates and Liouville theorems for harmonic maps},  Intern. J. Math., 11 No.3, (2000) 413-448

\bibitem[PS]{PS}  P. Price and L. Simon, {\it Monotonicity formulae for Harmonic maps and Yang-Mills fields}, preprint, Canberra 1982. Final verson by P. Price, {\it A monotonicity formula for Yang-Mills fields},
Manus. Math. 43 (1983), 131-166

\bibitem[Sa1]{Sa1}  I.M. Salavessa, {\it Graphs with parallel mean curvature},
Proc. A.M.S. Vol. 107, No.2, (1989), 449-458

\bibitem[Sa2]{Sa2}  I.M. Salavessa, {\it Spacelike graphs with parallel mean
curvature}, Bull. Bel. Math. Soc. 15 (2008), 65-76

\bibitem[San]{San}  A. Sanini, {\it Applicazioni tra variet`a riemanniane con
energia critica rispetto a deformazioni di metriche}, Rend. Mat., 3 (1983),
53-63

\bibitem[Se1]{Se1}  H.C.J. Sealey, {\it Some conditions ensuring the vanishing of
harmonic differential forms with applications to harmonic maps and
Yang-Mills theory}, Math. Soc. Camb. Phil. Soc. 91 (1982), 441-452

\bibitem[Se2]{Se2}  H.C.J. Sealey, {\it The stress energy tensor and vanishing of
$L^2$ harmonic forms}, preprint

\bibitem[SiSiYa]{SiSiYa}  L. Sibner, R. Sibner and Y.S. Yang, {\it Generalized
Bernstein property and gravitational strings in Born-Infeld theorey},
Nonlinearity \textbf{20} (2007) 1193-1213

\bibitem[Si]{Si}  J. Simons, {\it Minimal varieties in Riemannian manifolds}, Ann.
of Math. (2) 88 (1968), 62-105

\bibitem[SU]{SU}  R. Schoen, K. Uhlenbeck, {\it A regularity theory for harmonic
maps}, J. Diff. Geom. 17 (1982), 307-335

\bibitem[We1]{We1}  S.W. Wei, {\it On 1-harmonic functions}, SIGMA Symmetry Integrability Geom. Methods Appl. 3 (2007), Paper 127, 10 pp. arXiv:0712.4282

\bibitem[We2]{We2}  S.W. Wei, {\it $p$-Harmonic geometry and related topics}. Bull. Transilv. Univ. Brasov Ser. III 1(50) (2008), 415-453

\bibitem[We3]{We3} S. W. Wei, {\it  Representing homotopy groups and spaces of maps by $p$-harmonic
maps}, Indiana Univ. Math. J., {\bf 47},  (1998), 625--670.

\bibitem[Xi1]{Xi1}  Y.L. Xin, {\it Differential forms, conservation law and
monotonicity formula}, Scientia Sinica (Ser A) Vol. XXIX (1986), 40-50

\bibitem[Xi2]{Xi2}  Y.L. Xin, {\it On Gauss image of a spacelike hypersurface
with constant mean curvature in Minkowski space}, Comm. Math. Helv. 66 (1991)
590-598

\bibitem[Xi3]{Xi3}  Y.L. Xin, {\it A rigidity theorem for a space-like graph of
higher codimension}, Manus. Math. 103 (2) (2000) 191-202.

\bibitem[Ya]{Ya}  Y.S. Yang, {\it Classical solutions in the Born-Infeld theory},
Proc. R. Soc. Lond. A \textbf{456} (2000), 615-640
\end{thebibliography}
\enddocument